\documentclass[a4paper,12pt]{article}
\usepackage{graphicx}    
\usepackage{stmaryrd}
\usepackage{amssymb} 
\usepackage{amsmath}
\usepackage{amsfonts}
\usepackage{amsthm} 
\usepackage{mathrsfs}
\usepackage{enumerate} 
\usepackage{authblk} 
\usepackage{color} 

\title{Linear-Quadratic Mean Field Games}
\date{}
\author[$\star$]{A. Bensoussan}
\author[$*$, $\dagger$]{ K. C. J. Sung}
\author[$\ddagger$]{ S. C. P. Yam}
\author[$*$]{ S. P. Yung}
\affil[$\star$]{International Center for Decision and Risk Analysis\\ School of Management\\ The University of Texas at Dallas}
\affil[$\star$]{Graduate School of Business and Department of Logistics and Maritime Studies\\The Hong Kong Polytechnic University}
\affil[$\star$]{Graduate Department of Financial Engineering, Ajou University}
\affil[$*$]{Department of Mathematics\\ The University of Hong Kong}
\affil[$\dagger$]{Department of Statistics and Actuarial Science\\ The University of Hong Kong}
\affil[$\ddagger$]{Department of Statistics\\ The Chinese University of Hong Kong}
\begin{document}

\newtheorem{thm}{Theorem}[section]
\newtheorem{cor}[thm]{Corollary}
\newtheorem{defn}[thm]{Definition}
\newtheorem{lem}[thm]{Lemma}
\newtheorem{problem}[thm]{Problem}
\newtheorem{prop}[thm]{Proposition}
\newtheorem{rk}[thm]{Remark}
\setcounter{section}{0}

\maketitle
\begin{abstract} As an organic combination of mean field theory in statistical physics and (non-zero sum) stochastic differential games, Mean Field Games (MFGs) has become a very popular research topic in the fields ranging from physical and social sciences to engineering applications, see for example the earlier studies by Huang, Caines and Malham\'{e} (2003), and that by Lasry and Lions (2006a, b and 2007). In this paper, we provide a comprehensive study of a general class of mean field games in the linear quadratic framework. We adopt the adjoint equation approach to investigate the existence and uniqueness of equilibrium strategies of these Linear-Quadratic Mean Field Games (LQMFGs). Due to the linearity of the adjoint equations, the optimal mean field term satisfies a forward-backward ordinary differential equation. For the one dimensional case, we show that the equilibrium strategy always exists uniquely. For dimension greater than one, by choosing a suitable norm and then applying the Banach Fixed Point Theorem, a sufficient condition for the unique existence of the equilibrium strategy is provided, which is independent of the coefficients of controls and is always satisfied whenever those of the mean-field term are vanished (and therefore including the classical Linear Quadratic Stochastic Control (LQSC) problems as special cases). As a by-product, we also establish a neat and instructive sufficient condition, which is apparently absent in the literature (see Freiling (2002)) and only depends on coefficients, for the unique existence of the solution for a class of non-trivial nonsymmetric Riccati equations. Numerical examples of non-existence of the equilibrium strategy will also be provided. It is remarked that the uniform agent case of Huang, Caines and Malham\'{e} (2007a) serves as an interesting comparison with our LQMFGs. {\color{red}We give an example (see Appendix) with which existence can be covered by our theory while it needs not satisfy the sufficient condition provided in their work; though in general, both approaches cover different feasible ranges.} Finally, similar approach has been adopted to study the Linear-Quadratic Mean Field Type Stochastic Control Problems (see Andersson and Djehiche (2010)) and their comparisons with MFG counterparts. \\

\end{abstract}

\section{Introduction}
Modeling collective behaviors of individuals in account of their mutual interactions in various physical or sociological dynamical systems has been one of the major problems in the history of mankind. For instance, physicists were simply used to apply the traditional variational methods from Lagrangian or Hamiltonian mechanics to study interacting particle system, which left a drawback of extremely high computational cost that made this microscopic approach almost mathematically intractable. To resolve this matter, a completely different macroscopic approach from statistical physics had been gradually developed, which eventually leads to the primitive notion of mean field theory. The novelty of this approach is that particles interact through a medium, namely the mean field term, being aggregated by action of and reaction on each particle. Moreover, by passing the number of particles to the infinity in these macroscopic models, the mean field term will become a functional of the density function which represents the whole population of particles that leads to much less computational complexity. In biological literature, similar tools have been applied to connect human interactive motion with herding models for insects and animals. For example, the behavior that ants secrete chemical substrates for leading mates to valuable food resources resulting in a lane can be described by a mean-field model (see Kirman~\cite{Ki} for more details).

\smallskip

On economics side, due to the dramatic population growth and rapid urbanization, urgent needs of in-depth understanding of collective strategic interactive behaviors of a huge group of investors is crucial to maintaining sustainable economic growth. Since the vector of good prices is determined by both demand and supply, it is natural to utilize the aggregation effect from the investors' states as a canonical candidate of mean-field term, and then employs the corresponding mean-field models in place of the classical equilibrium models in economics; moreover, as the investors are usually smart in decision making (i.e. being not of zero-intelligent), it is necessary to also incorporate the theory of stochastic differential games (SDGs) in these mean-field models. Over the past few decades, SDGs has been a major research topic in control theory and financial economics, especially in studying the continuous-time decision making problem between non-cooperative investors; in regard to the one-dimensional setting the theory of two person zero-sum games is quite well-developed via the notion of viscosity solutions, see for example Elliott (1976), and Fleming and Souganidis (1989). Unfortunately, most interesting SDGs are $N$-player non-zero sum SDGs. In this direction we mention the  works of Bensoussan and Frehse~\cite{BF1, BF2} and Bensoussan et al.~\cite{BFV}, but there are still relatively few results in the literature.

\smallskip

{\color{red}As a macroscopic equilibrium model, Huang et al.~\cite{HCM1, HCM2} investigated stochastic differential game problems involving infinitely many players under the name ``Large Population Stochastic Dynamic Games''. Independently, Lasry and Lions~\cite{LL1, LL2, LL3} introduced studied similar problems from the viewpoint of the mean-field theory and termed ``Mean-Field Games (MFGs)''. As an organic combination of mean field theory and theory of stochastic differential games, MFGs provide more realistic interpretation of individual dynamics at the microscopic level, so that each player are not of zero-intelligent and will be able to strategically optimize his prescribed objectives, yet with a mathematical tractability in a macroscopic framework.} To be more precise, the general theory of MFGs has been built by combining various consistent assumptions on the following modeling aspects: (1) a continuum of players; (2) homogeneity in strategic performance of players; and (3) social interactions through the impact of mean field term. The first aspect is describing the approximation of a game model with a huge number of players by a continuum one yet with a sufficient mathematical tractability. The second aspect is assuming that all players obey the same set of rules of the interactive game, which provide guidance on their own behavior that potentially leads them to ultimate success. Finally, due to the intrinsic complexity of the society in which the players participate in, the third aspect is explaining the fact that each player is so negligible who can only affect others marginally through his own infinitesimal contribution to the society. In a MFG, each player will base his decision making purely on his own criteria and certain summary statistics (that is, the mean field term) about the community; in other words, in explanation of their interactions, the pair of personal and mean-field characteristics of the whole population is already sufficient and exhaustive. Mathematically, each MFG will possess the following forward-backward structure: (1) a forward dynamic describes the individual strategic behavior; (2) a backward equation describes the evolution of individual optimal strategy, such as those in terms of the individual value function via the usual backward recursive techniques. For the detail of the derivation of this system of equations with forward-backward feature, one can consult from the works of Huang et al.~\cite{HCM2} and Lasry and Lions~\cite{LL1,LL2,LL3}.

\smallskip

Before introducing our proposed model, we first list out some relevant recent theoretical results in MFGs. (I) For problems over infinite-time horizon: Bardi~\cite{BP}, and Li and Zhang~\cite{LZ} studied ergodic MFGs with different quadratic cost functionals and linear dynamics; Gu\'{e}ant~\cite{Gu1} studied a specific MFG with a quadratic Hamiltonian and showed that the density function for the population is Gaussian; Huang et al.~\cite{HCM1, HCM3, HCM6, HCM7} considered MFGs with quadratic cost functional; Nourian et al.~\cite{NCM} extends the studies of MFGs with ergodic cost functional to Cucker-Smale Flocking model; and Yin et al.~\cite{YMMS} gave a bifurcation analysis of an ergodic MFG with nonlinear dynamics. (II) For problems over finite-time horizon: Gu\'{e}ant~\cite{Gu2} applied a change-of-variable technique leading to separation-of-variables to consider MFGs with quadratic Hamiltonian; Lachapelle~\cite{Lac} and Lachapelle and Wolfram~\cite{LW} extended MFGs to the setting involving 2-population dynamics; Lachapelle et al.~\cite{LST} investigated MFGs with reflection parts and quadratic cost functional; Tembine et al.~\cite{TZB} considered the risk-sensitive MFGs; and Yang et al.~\cite{YMM} adopts the MFG approach to construct a non-linear filter. (III) Various numerical approximation scheme can also be found in Achdou et al.~\cite{ACCD}, and Achdou and Capuzzo-Dolcetta~\cite{AchCD}. Because of the discretization in the numerics, Gomes et al.~\cite{GMS} studied discrete time mean field game with finite state space directly. For more recent development and its applications, please also refer to the lecture notes Cardaliaguet~\cite{Car}, the surveys Gu\'{e}ant~\cite{GuT}, Gu\'{e}ant et al.~\cite{GLL}, and the references therein.

\smallskip

{\color{red}In this paper, we study a subclass of Mean Field Games in which the cost functional is quadratic in all state variables, control variables and the mean field terms; while the controlled dynamics are linear and also consist of mean field terms. These Linear-Quadratic mean field games (LQMFGs) have been previously considered in Huang et al.~\cite{HCM4} by using the common Riccati equation approach; in contrast, in this paper, the stochastic maximum principle is adopted instead. Essentially, the equilibrium problem can be converted into find a fixed point for a transformation defined by the solution of a control problem. For the part of control problem, these two approaches are certainly equivalent. However, it is not the case for the fixed point problem since a condition that is easier to be verified can be given{\footnote{This does not mean that our condition is less restrictive. In general, both approaches cover different feasible ranges.}}. Indeed, in our approach, thanks to the linearity of the adjoint equations, the optimal mean-field term can be expressed as the solution of a forward-backward ordinary differential equation.} This method avoids us from solving the optimal trajectory as in the Riccati equation approach and allows generalization to higher dimension, which is a crucial setting in understanding the 2-population MFGs proposed in Lachapelle~\cite{Lac} and Lachapelle and Wolfram~\cite{LW}. More precisely, by choosing a suitable norm and applying the Banach Fixed Point Theorem, we provide a more relaxed sufficient condition for the existence and uniqueness of the equilibrium strategy, which is also independent of the coefficient of control and always hold when the mean-field term is zero. Under the one-dimensional setting and certain convexity assumption, we also prove that the equilibrium strategy always uniquely exists. As a by-product, we also establish a neat and instructive sufficient condition, which is apparently absent in the literature (see Freiling~\cite{Fre}) and only depends on coefficients, for the unique existence of the solution for a class of non-trivial  \textit{\textbf{nonsymmetric}} Riccati equations. Numerical examples showing the non-existence of any equilibrium strategy will also be provided. {\color{red}Furthermore, in the Appendix, we compare explicitly our conditions with those of  Huang et al.~\cite{HCM4}. In summary, our present work gives a novel and totally different approach with several advantages in particular for the generalization of the classical Linear-Quadratic Stochastic Control Problem in the MFG setting.}

\smallskip

In general, the computational complexity of calibrating a Nash equilibrium of an $N$-player SDG (if it exists) is very high, especially for large values of $N$, it would be more convenient to find a computable approximation of this Nash Equilibrium strategy. Since MFGs are obtained by setting $N \to \infty$,  the equilibrium strategy serves as a natural candidate as it can be shown to be an ``approximation'', or $\epsilon$-Nash Equilibrium strategy for the corresponding equilibrium for $N$-player SDG. The computability of this equilibrium strategy is justifiable as it depends only on the state of the player and the mean-field term, which dramatically reduces the problem dimension of the Nash Equilibrium strategy of the $N$-player SDG. For more inspiring elaboration on the notion of $\epsilon$-Nash Equilbrium, one can refer to, for example, Cardaliaguet~\cite{Car} and Huang et al.~\cite{HCM1,HCM5,HCM2}.

\smallskip 

If one considers a centralized controlling interacting particle system, instead of every particle having the free will to choose its own control as formulated in MFGs, a stochastic control problem of mean field type would be resulted (see Andersson and Djehiche~\cite{AD}). This mean-field type optimization problem shares a similar mathematical form as proposed in MFG problem, and the mean-field term is now uniformly controlled by a centralizing system instead of being affected by the collective optimal trajectory. For more details about the existence and convergence rate of the related mean-field backward stochastic differential equations, one can refer to Buckdahn et al.~\cite{BDLP} and Buckdahn et al.~\cite{BLP}. By using the adjoint equation approach again, we characterize the optimal control, which exists and is unique in virtue of the convex coercive property of the underlying cost functional. Finally, we also find that, in general, this optimal control is different from the equilibrium strategy obtained in its corresponding MFG counterpart.

\smallskip

In Section~\ref{ProbForm}, we will formulate a Linear-Quadratic $N$-player nonzero-sum stochastic differential game and demonstrate how to obtain the corresponding mean field game formally. In Section~\ref{SolnMFG}, we shall employ the adjoint equation approach in order to provide a thoughtful study of the the existence and uniqueness of LQMFGs. By choosing a suitable norm and applying the Banach Fixed Point Theorem, an illuminating sufficient condition for the existence and uniqueness of the equilibrium strategy is provided; note that this new condition is independent of the coefficients of the controls, and is always satisfied whenever the coefficients of the mean-field terms are vanished. {\color{red}Relationship with nonsymmetric Riccati equations and illustrative numerical examples will also be provided. We remark that these nonsymmetric Riccati equations, appeared in the resolution of the fixed point problem, could not be found in the literature including  Huang et al.~\cite{HCM4}; and most importantly, they are substantially different from those symmetric Riccati equation commonly arisen from Control Theory.} In Section~\ref{EpsNash}, we shall show that the equilibrium strategy is an $\epsilon$-Nash Equilibrium of the $N$-player SDG. In Sections~\ref{MFTSC} and \ref{Compare}, we shall adopt a similar adjoint equation approach to solve the Linear-Quadratic Mean-Field Type Stochastic Control Problem and compare its optimal control to the equilibrium strategy of the corresponding MFG counterpart. {\color{red}In Appendix, an example is given  which illustrates that its unique existence could be covered by our theory but it fails to satisfy the sufficient condition as provided in Huang et al.~\cite{HCM4}. It is noticed that in some other cases, Huang et al.~\cite{HCM4} may cover different possibilities from ours.}

\section{Problem Formulation}\label{ProbForm}
The present formulation of the Linear-Quadratic Mean Field Games will follow closely the classical Linear-Quadratic Stochastic Control Problems, see for example Bensoussan~\cite{B}. Following Lasry and Lions~\cite{LL1, LL2, LL3}, in order to formulate the Linear-Quadratic Mean Field Game, we first state the corresponding $N$-player game for $N \geq 1$.

\smallskip

Let $(\Omega, \mathcal{F}, \mathbb{P})$ be a complete probability space and $T > 0$ be the time horizon. Suppose that $W^1, \ldots, W^N$ are $N$ independent $n$-dimensional standard Wiener processes defined on $(\Omega, \mathcal{F}, \mathbb{P})$ and $x_0^1, \ldots, x_0^N$ are $N$ independent, identically distributed (i.i.d.) $n$-dimensional random vectors. We also assume that $x^i_0$ is independent to $(W^1, \ldots, W^N)$ for each $i$, $1 \leq i \leq N$. The dynamics of the player $i$ is modeled by
\[
dx^i_t = \left(A_tx^i_t + B_tv^i_t + \bar{A_t}\cdot\frac{1}{N-1}\sum_{j = 1, j \neq i}^Nx^j_t\right) dt + \sigma_t\,dW^i_t, \quad
x^i(0) = x_0^i,
\]
where $A$, $B$, $\bar{A}$ are bounded deterministic matrix-valued functions in time of suitable sizes, $\sigma$ is a $L^2$-function in time of suitable size, and the control $v^i$ in $L^2_{\mathcal{G}}(0, T; \mathbb{R}^m)$, which is the $L^2$-space of stochastic processes adapted to the filtration $\mathcal{G}_t \triangleq \sigma((x^1_0, \ldots, x^N_0), (W^1_s, \ldots, W^N_s), s \leq t)$, with values in $\mathbb{R}^m$. The present proposed (additive) model extends the classical linear stochastic dynamical one, as expected, the coefficient $A$ measures the effect brought by the state variable and $B$ measures the impact of the control; while the new ingredient $\bar{A}$ summarizes the symmetric influence of the rest of the players. Even though this additional term $\bar{A}$ is natural from the modeling perspective when one attempts to investigate interactive real-time multi-player game, it causes a substantial mathematical difficulty in the discussion on the existence and calibration of the corresponding Nash Equilibrium  especially for large values of $N$. Although the coefficients look the same for all players, it reflects that each players obeys the same set of game rules and behaves based on the same collections of rationales. Their actual individual realized performances could be far different from each other, in particular, their single dynamics are driven by independent Wiener processes.

\smallskip

The cost functional for each player $i$ is assumed to be:
\begin{align*}
& \mathcal{J}^i(v^1, \ldots, v^N)\\
\triangleq&\,\, \mathbb{E}\left[\frac{1}{2}\int^T_0 (x^i_t)^* Q_t x^i_t + (v^i_t)^* R_tv^i_t\,dt + \frac{1}{2}(x^i_T)^* Q_Tx^i_T \right]\\
&+ \mathbb{E}\left[\frac{1}{2}\int^T_0 \left(x^i_t - S_t \cdot \frac{1}{N-1}\sum_{j = 1, j\neq i}^Nx^j_t\right)^* \bar{Q}_t\left(x^i_t - S_t \cdot \frac{1}{N-1}\sum_{j = 1, j \neq i}^Nx^j_t\right) \,dt \right]\\
&+ \mathbb{E}\left[\frac{1}{2} \left(x^i_T - S_T \cdot \frac{1}{N-1}\sum_{j = 1, j \neq i}^Nx^j_T\right)^* \bar{Q}_T\left(x^i_T - S_T \cdot \frac{1}{N-1}\sum_{j = 1, j \neq i}^Nx^j_T\right) \right],
\end{align*}
where $M^*$ denotes the transpose of a matrix $M$, $S$ (respectively, $Q$, $\bar{Q}$ and $R$) are bounded, deterministic (respectively, non-negative and positive definite) matrix-valued functions in time of suitable sizes. We also suppose that $R \geq \delta I$, for some $\delta > 0$.

\smallskip

The reasons for the same form of the objective functional among different players are the same as in the previous discussions on the modeling of individual dynamics. The first expectation agrees with the corresponding cost functional in the classical Linear-Quadratic Stochastic Control Problem; namely, it describes the sum of running expenses and the terminal costs of each player himself. The other two expectations are specific in our present model setting, they describe the extra costs incurred if a player shows deviated performance away from the average behavior of the community. These two terms truly reflect the coalescence phenomena commonly observed in the literature of socio-economics and finance, in the sense that, every agent has to pay an additional transaction cost of collecting extra profitable information if he aims to outperform from his peers. Along every direction of the deviation, the incorporated additional cost has to be non-negative, this model constraint is ensured by the assumption of the non-negative definiteness of $\bar{Q}$. 

\smallskip

The principal objective of each player is to minimize his own cost functional by properly controlling his own dynamics. In this classical non-zero sum stochastic differential game framework, we aim to establish a Nash equilibrium $(u^1, \ldots, u^N)$ (see for example, Bensoussan and Frehse~\cite{BF1}):
\begin{problem}\label{FiniteDim}Find a Nash Equilibrium $(u^1, \ldots, u^N)$ which satisfies the following comparison inequalities:
\[
\mathcal{J}^i(u^1, \ldots, u^{i-1}, v^i, u^{i+1}, \ldots u^N) \geq \mathcal{J}^i(u^1, \ldots, u^N),
\]
for $1 \leq i \leq N$ and any admissible control $v^i$ in $L^2_{\mathcal{G}}(0,T;\mathbb{R}^m)$.
\end{problem}

\smallskip

In accordance with the permutation symmetry of index $i$, it suffices to consider the case for $i = 1$. In general, the computational complexity of calibration of a Nash equilibrium (if it exists) is high, especially for large values of $N$. Due to the large number of participants in most game theoretical models in practice, a convenient computable approximation of the Nash Equilibrium strategy is usually demanded. By formally passing $N \to \infty$, Lasry and Lions~\cite{LL1,LL2,LL3} introduced the notion of Mean Field Games. Now, the Mean Field Game associated with Problem~\ref{FiniteDim} can be obtained as follows:

\begin{problem}\label{MFG}
Find an equilibrium strategy $u$ in $L^2_{\mathcal{F}}(0, T; \mathbb{R}^m)$, with $x_0 \triangleq x^1_0$, $W \triangleq W^1$ and $\mathcal{F}_t \triangleq \sigma(x_0, W_s, s \leq t)$, which minimizes the cost functional
\begin{align*}
J(v) \triangleq&\,\, \mathbb{E}\left[\frac{1}{2}\int^T_0 x_t^* Q_t x_t + v_t^* R_tv_t  + (x_t - S_t\mathbb{E}[y_t])^* \bar{Q}_t(x_t - S_t\mathbb{E}[y_t]) \,dt \right]\\
&+ \mathbb{E}\left[\frac{1}{2} x_T^* Q_Tx_T + \frac{1}{2}(x_T - S_T\mathbb{E}[y_T])^* \bar{Q}_T(x_T - S_T\mathbb{E}[y_T])\right],
\end{align*}
where the dynamics is given by 
\[
dx_t = \left(A_tx_t + B_tv_t + \bar{A_t}\,\mathbb{E}[y_t]\right)\,dt + \sigma_t\,dW_t, \quad x(0) = x_0,
\]
$v$ is an admissible control in $L^2_{\mathcal{F}}(0, T; \mathbb{R}^m)$ and $y$ is the trajectory corresponding to the equlibrium strategy $u$ (if it exists). 
\end{problem}

\begin{rk}An interesting example of our proposed one-dimensional LQMFGs has been considered earlier in Huang et al.~\cite{HCM4}. In this paper, we shall provide a complete picture of the resolution of the problem by using adjoint equation approach, which results in a different  sufficient condition for the unique existence of the underlying equilibrium strategy.
\end{rk}

\smallskip

In comparison with the $N$-player game, in order to avoid confusion, the notations for the dynamics $x$ and $y$ stated in Problem~\ref{MFG} will be changed to $\hat{x}$ and $\hat{y}$ respectively. For if Problem~\ref{MFG} were solvable, then for each $i$, $1 \leq i \leq N$, we could obtain a strategy $u^i$ in $L^2_{\mathcal{F}^i}(0, T; \mathbb{R}^m)$, where $\mathcal{F}^i_t \triangleq \sigma(x^i_0, W^i_s, s \leq t)$. Since $\mathbb{E}[y_t]$ is a deterministic process, it is clear that $u^1, \ldots, u^N$ are i.i.d.. As the Mean Field Game is obtained from the $N$-player game, it is expected that $(u^1, \ldots, u^N)$ is an $\epsilon$-Nash Equilibrium when $N \to \infty$, see for example Cardaliaguet~\cite{Car} and Huang et al.~\cite{HCM1,HCM5,HCM2} for more detail. We shall first present an informal description here, rigorous arguments will be provided later in Section~\ref{EpsNash}. Again, it suffices to consider Player 1.

\smallskip

For any admissible control $v^1$, let $(x^1, \ldots, x^N)$ (respectively, $(y^1, \ldots, y^N)$) denote the dynamics in Problem~\ref{FiniteDim} controlled by $(v^1, u^2, \ldots, u^N)$ (respectively, $(u^1, \ldots, u^N)$). By the definition of $\epsilon$-Nash Equilibrium, $u^1$ will ``approximately'' minimize $\mathcal{J}^1(v^1, u^2, \ldots, u^N)$. As $N \to \infty$, for $i \neq 1$, we have 
\[
\frac{1}{N-1}\sum_{j = 1, j \neq i}^N x^j_t - \frac{1}{N-1}\sum_{j = 2, j \neq i}^N x^j_t \to 0.
\]
By the McKean-Vlasov argument, $x^i_t \to \hat{y}^i_t$. As an application of the Strong Law of Large Numbers (SLLN), $x^1_t \to \hat{x}^1_t$ as $\hat{y}^1_t, \ldots, \hat{y}^N_t$ are i.i.d., which is a consequence of the i.i.d. nature of $u^1, \ldots, u^N$. By applying SLLN again to the cost functional, we deduce that
\[
\mathcal{J}^1(v^1, u^2, \ldots, u^N) \to J^1(v^1).
\]
Similarly, we also have
\[
\mathcal{J}^1(u^1, u^2, \ldots, u^N) \to J^1(u^1),
\]
and it shows heuristically that $(u^1, \ldots, u^N)$ is an $\epsilon$-Nash Equilibrium.

\section{Solution of the Mean Field Game}\label{SolnMFG}
To motivate for solving Problem~\ref{MFG}, we first lay down some classical results in the literature of Linear-Quadratic Stochastic Control Theory but from a new perspective which aids for the development of our new methodology to tackle Problem~\ref{MFG}.

\begin{problem}\label{SLQ} Given a continuous deterministic process $z$ with values in $\mathbb{R}^n$. Find an optimal control $u$ in $L^2_{\mathcal{F}}(0, T; \mathbb{R}^m)$ which minimizes
\begin{align*}
J(v) \triangleq&\,\, \mathbb{E}\left[\frac{1}{2}\int^T_0 x_t^*Q_tx_t + v_t^*R_tv_t + (x_t - S_tz_t)^*\bar{Q}_t(x_t - S_tz_t)\,dt\right]\\
&+\mathbb{E}\left[ \frac{1}{2}x_T^*Q_Tx_T + \frac{1}{2}(x_T - S_Tz_T)^*\bar{Q}_T(x_T - S_Tz_T)\right],
\end{align*}
where the dynamics is given by
\[
dx_t = (A_tx_t + B_tv_t + \bar{A}_tz_t) \,dt + \sigma_t\,dW_t, \quad x(0) = x_0,
\]
and $v$ is an admissible control in $L^2_{\mathcal{F}}(0,T;\mathbb{R}^m)$.
\end{problem}
\begin{thm}\label{ThmSLQ}
Problem~\ref{SLQ} is uniquely solvable and the optimal control $u$ is $-R^{-1}B^*p$,  where $(y, p)$ satisfy the stochastic maximum principle relation
\begin{equation}\label{stocmax}
\left\{
\begin{aligned}
dy_t &= (A_ty_t - B_tR^{-1}_tB^*_tp_t + \bar{A}_tz_t)\,dt + \sigma_t\,dW_t,\\
y_0 &= x_0,\\
-\frac{d\omega_t}{dt} &= A^*_t\omega_t + (Q_t + \bar{Q}_t)y_t - (\bar{Q}_tS_t)z_t,\\
\omega_T &= (Q_T + \bar{Q}_T)y_T - (\bar{Q}_TS_T)z_T,
\end{aligned}
\right.
\end{equation}
such that $p_t = \mathbb{E}[\omega_t | \mathcal{F}_t]$.
\end{thm}
\begin{proof}
It is clear that Problem~\ref{SLQ} is a strictly convex coercive optimization problem in the sense that $J(v) \to \infty$ as $\|v\|\to \infty$. In order to derive the stochastic maximum principle relation, we first consider the Euler Equation:
\[
\frac{d}{d\theta}J(u(\cdot) + \theta\,v(\cdot))\bigg|_{\theta = 0} = 0.
\]
To explicitly express the dependence of the state $x(\cdot)$ on $v(\cdot)$ (recall that $z(\cdot)$, $z_T$ are fixed), we adopt the notation $x(t; v(\cdot))$. Note that, by linearity,
\[
x(t; u(\cdot)+\theta v(\cdot)) = y(t) +\theta\tilde{x}(t;v(\cdot)),
\]
where
\begin{equation*}
\frac{d\tilde{x}}{dt} = A_t\tilde{x}_t + B_tv_t, \quad \tilde{x}(0) = 0.
\end{equation*}

\smallskip

On the other hand, we can write $J(v(\cdot))$ as 
\begin{align*}
J(v(\cdot)) =&\,\, \mathbb{E}\left[\int^T_0\frac{1}{2}(x^*_t(Q_t + \bar{Q}_t)x_t + v^*_tR_tv_t) - x^*_t(\bar{Q}_tS_t)z_t + \frac{1}{2}z^*_t(S^*_t\bar{Q}_tS_t)z_t\,dt\right]\\
&+\mathbb{E}\left[\frac{1}{2}x^*_T(Q_T + \bar{Q}_T)x_T - x^*_T(\bar{Q}_TS_T)z_T + \frac{1}{2}z^*_T(S^*_T\bar{Q}_TS_T)z_T\right],
\end{align*}
and therefore the Euler Equation becomes
\begin{multline}\label{EulerCond}
\mathbb{E}\left[\int^T_0\tilde{x}^*_t(Q_t + \bar{Q}_t)y_t - \tilde{x}^*_t(\bar{Q}_tS_t)z_t + v^*_tR_tu_t\,dt\right]\\
+ \mathbb{E}[\tilde{x}^*_T(Q_T + \bar{Q}_T)y_T - \tilde{x}^*_T(\bar{Q}_TS_T)z_T] = 0.
\end{multline}

\smallskip

Define the adjoint process $\omega_t$ (need not adapted to $\mathcal{F}_t$) by 
\begin{equation*}
\left\{
\begin{aligned}
-\frac{d \omega_t}{dt} &= A^*_t\omega_t + (Q_t + \bar{Q}_t)y_t - (\bar{Q}_tS_t)z_t,\\
\omega_T &= (Q_T + \bar{Q}_T)y_T - (\bar{Q}_TS_T)z_T,
\end{aligned}
\right.
\end{equation*}
we obtain
\begin{equation*}
\begin{aligned}
\frac{d}{dt}(\tilde{x}^*_t\omega_t) &= (\tilde{x}^*_tA^*_t+v^*_tB^*_t)\omega_t - \tilde{x}^*_t(A^*_t\omega_t + (Q_t + \bar{Q}_t)y_t - (\bar{Q}_tS_t)z_t)\\
&= v^*_tB^*_t\omega_t - \tilde{x}^*_t((Q_t + \bar{Q}_t)y_t - (\bar{Q}_tS_t)z_t),
\end{aligned}
\end{equation*}
and hence
\begin{align*}
&\,\, \mathbb{E}\left[\int^T_0v^*_tB^*_t\omega_t\,dt\right]\\
=&\,\, \mathbb{E}[\tilde{x}^*_T((Q_T + \bar{Q}_T)y_T - (\bar{Q}_TS_T)z_T)] + \mathbb{E}\left[\int^T_0\tilde{x}^*_t((Q_t + \bar{Q}_t)y_t - (\bar{Q}_tS_t)z_t)\,dt\right],
\end{align*}
and from (\ref{EulerCond}) we obtain
\begin{equation}\label{ArbCond}
\mathbb{E}\left[\int^T_0v^*_t(B^*_t\omega_t + R_tu_t)\,dt\right] = 0.
\end{equation}
Set $p_t = \mathbb{E}[\omega_t | \mathcal{F}_t]$ then (\ref{ArbCond}) becomes $\mathbb{E}[\int^T_0v^*_t(B^*_tp_t + B_tu_t)\,dt] = 0$. Since $v$ is arbitrary in $\mathcal{L}^2_{\mathcal{F}}(0, T; \mathbb{R}^m)$, we deduce that $u = -R^{-1}B^*p$.
\end{proof}

\begin{rk}\label{RicControl}
The optimal control $u$ can be written as $-R^{-1}B^*(\Xi y + \zeta)$, where $\Xi$ and $\zeta$ satisfies
\begin{equation*}
\left\{
\begin{aligned}
\frac{d\Xi_t}{dt} &+ \Xi_tA_t + A^*_t\Xi_t - \Xi_t(B_tR^{-1}_tB^*_t)\Xi_t + Q_t + \bar{Q}_t = 0,\\
\Xi_T &= Q_T + \bar{Q}_T,\\
\frac{d\zeta_t}{dt} &= -A^*_t\zeta_t + \Xi_t(B_tR_t^{-1}B_t^*)\zeta_t + (\bar{Q}_tS_t - \Xi_t\bar{A}_t)z_t,\\
\zeta_T &= -\bar{Q}_TS_T.
\end{aligned}
\right.
\end{equation*}
\end{rk}

\smallskip

According to Theorem~\ref{ThmSLQ}, for fixed $z$, we shall obtain the pair $(y, p)$ and the optimal control is $u = -R^{-1}B^*p$. Hence $u$ is an equilibrium strategy of Problem~\ref{MFG} if and only if there is a continuous function $z$ such that $\mathbb{E}[y] = z$. We denote the expected values $\bar{y}_t \triangleq \mathbb{E}[y_t]$, $\bar{p}_t \triangleq \mathbb{E}[\omega_t] = \mathbb{E}[p_t]$, the stochastic maximum principle relation (\ref{stocmax}) implies that
\begin{equation*}
\left\{
\begin{aligned}
\frac{d}{dt}
\begin{pmatrix}
\bar{y}_t \\ -\bar{p}_t
\end{pmatrix}
&=
\begin{pmatrix}
A_t  & -B_tR_t^{-1}B_t^*\\
Q_t + \bar{Q}_t & A^*_t
\end{pmatrix}
\begin{pmatrix}
\bar{y}_t \\ \bar{p}_t
\end{pmatrix}
+
\begin{pmatrix}
\bar{A}_tz_t\\
-(\bar{Q}_tS_t)z_t
\end{pmatrix},\\
\bar{y}_0 &= \mathbb{E}[x_0], \\
\bar{p}_T &= (Q_T + \bar{Q}_T)\bar{y}_T - (\bar{Q}_TS_T)z_T.
\end{aligned}
\right.
\end{equation*}
Hence, Problem~\ref{MFG} is solvable if and only if $(\xi, \eta) = (\bar{y}, \bar{p})$ solves the following system of ordinary differential equations: 
\begin{equation}\label{NewRic}
\left\{
\begin{aligned}
\frac{d}{dt}
\begin{pmatrix}
\xi_t \\ -\eta_t
\end{pmatrix}
&=
\begin{pmatrix}
A_t + \bar{A}_t  & -B_tR_t^{-1}B_t^*\\
Q_t  + \bar{Q}_t(I - S_t)& A^*_t
\end{pmatrix}
\begin{pmatrix}
\xi_t \\ \eta_t
\end{pmatrix},\\
\xi_0 &= \mathbb{E}[x_0],\\
\eta_T &= (Q_T + \bar{Q}_T(I - S_T))\xi_T,
\end{aligned}
\right.
\end{equation}
where $I$ is the identity matrix.

\smallskip

We next address the uniqueness issue of the Nash Equilibrium. According to the previous discussion, if Equation (\ref{NewRic}) has at most one solution, at most one possible $\mathbb{E}[y_t]$ could be found. Together with the uniqueness result in Theorem~\ref{ThmSLQ}, there is at most one equilibrium strategy $u$. Conversely, suppose that there is at most one equilibrium strategy in Problem~\ref{ThmSLQ}. For each solution $(\xi, \eta)$ of Equation (\ref{NewRic}), we can associate the corresponding $y$ and $p$ with $z = \xi$. By the construction, we have $\mathbb{E}[y] = \xi$ and $\mathbb{E}[p] = \eta$ and hence $-R^{-1}B^*p$ is an equilibrium strategy. Due to the uniqueness of equilibrium strategy, there is at most one possible choice for $-R^{-1}B^*\,\mathbb{E}[p]$. Hence, the solution $\xi$ is uniquely determined by
\[
\frac{d\xi_t}{dt} = (A_t + \bar{A}_t)\xi_t - B_tR_t^{-1}B_t^*\,\mathbb{E}[p_t], \quad \xi_0 = \mathbb{E}[x_0].
\]
Similarly, $\eta$ can also be uniquely determined and the uniqueness result follows. The following theorem summarizes the previous discussion.

\begin{thm}\label{NesSuf}
There is a (unique) equilibrium strategy $u$ of Problem~\ref{MFG} if and only if there is a (unique) pair $(\xi, \eta)$ of the following system of ordinary differential equations (\ref{NewRic}):
\begin{equation*}
\left\{
\begin{aligned}
\frac{d}{dt}
\begin{pmatrix}
\xi_t \\ -\eta_t
\end{pmatrix}
&=
\begin{pmatrix}
A_t + \bar{A}_t  & -B_tR_t^{-1}B_t^*\\
Q_t + \bar{Q}_t(I - S_t) & A^*_t
\end{pmatrix}
\begin{pmatrix}
\xi_t \\ \eta_t
\end{pmatrix},\\
\xi_0 &= \mathbb{E}[x_0],\\
\eta_T &= (Q_T + \bar{Q}_T(I - S_T))\xi_T.
\end{aligned}
\right.
\end{equation*}
Moreover, this equilibrium condition depends on $\bar{Q}$ and $S$ only through $\mathcal{S} \triangleq \bar{Q}(I - S)$.
\end{thm}

\smallskip

Since $S$ could be quite arbitrary, $-BR^{-1}B^*$, $Q + \bar{Q}(I-S)$ may not often be of opposite sign and hence the sinusoidal functions serve as the standard example that Equation (\ref{NewRic}) does not admit a solution if $T$ is sufficiently large. Even under the assumption of the convexity of $Q + \bar{Q}(I-S)$, Equation (\ref{NewRic}) is still not in a canonical form commonly found in the literature of the classical optimal  control theory and the existence of solution is not guaranteed in our present context. To overcome this hurdle, we first define 
\begin{align*}
L &\triangleq T(\|Q_T + \mathcal{S}_T\|^2 + \|Q + \mathcal{S}\|_T)\|BR^{-1}B^*\|_T\\
&\,\,\cdot \exp((2\|A + \bar{A}\|_T + 2\|A^*\|_T + \|BR^{-1}B^*\|_T + \|Q + \mathcal{S}\|_T)T),
\end{align*}
where $\|M\|_T$ denotes the supremum norm of the deterministic matrix-valued function $M$ on $[0, T]$. By applying Gronwall's inequality and Banach Fixed Point Theorem, we first have the following standard existence result when $T$ is sufficiently small, see for example Ma and Yong~\cite{MY}:

\begin{prop}If $L < 1$, then there exists a unique solution of (\ref{NewRic}).
\end{prop}

However, in general, the supremum norm of $BR^{-1}B^*$ could be large and the condition that $L < 1$ is too restrictive. By the specific form of Equation (\ref{NewRic}), a more relaxed condition can be provided as follows:

\begin{thm}\label{MainThm}
Assume that the matrix-valued function $Q_t$ is invertible. Let $\phi(t,s)$ be the fundamental solution associated with $A_t$ and
\begin{align*}
\interleave\phi\interleave_{T} &\triangleq \sup_{0 \leq t \leq T}\sqrt{\left\|\phi^*(T,t)Q^{1/2}_T\right\|^2 + \int^T_t\left\|\phi^*(s,t)Q^{1/2}_s\right\|^2\,ds}.
\end{align*}
Also, we define that $\interleave\bar{A}\interleave_T \triangleq \sup_{0 < t < T}\left\|\bar{A}_tQ^{-1/2}_t\right\|$ and $\interleave\mathcal{S}\interleave_T \triangleq \sup_{0 < t \leq T}\left\|Q^{-1/2}_t\bar{Q}_t (I - S_t)Q^{-1/2}_t\right\|$. Suppose that $\interleave\bar{A}\interleave_T < \infty$, $\interleave\mathcal{S}\interleave_T < \infty$ and
\[
\sqrt{T}\interleave\phi\interleave_T \interleave\bar{A}\interleave_T (1 + \interleave\mathcal{S}\interleave_T) + \interleave\mathcal{S}\interleave_T < 1.
\]
Then there exists a unique solution of (\ref{NewRic}).
\end{thm}
\begin{proof}
Let $L^2_Q(0,T; \mathbb{R}^m)$ be the Hilbert space of functions endowed with the inner product
\[
\langle z, z' \rangle_Q \triangleq z^*_TQ_Tz'_T + \int_0^Tz_t^*Q_tz'_t\,dt.
\]
Given a function $z$ in $L^2_Q(0,T;\mathbb{R}^m)$ and there is a pair $(\xi, \eta)$ satisfying
\begin{equation}\label{AuxMap}
\left\{
\begin{aligned}
\frac{d\xi_t}{dt} &= A_t\xi_t - B_tR^{-1}_tB^*_t\eta_t + \bar{A}_tz_t,\\
\xi_0 &= \mathbb{E}[x_0],\\
-\frac{d\eta_t}{dt} &= A^*_t\eta_t + Q_t\xi_t + \bar{Q}_t(I-S_t)z_t,\\
\eta_T &= Q_T\xi_T + \bar{Q}_T(I-S_T)z_T,
\end{aligned}
\right.
\end{equation}
since it corresponds to a well-defined control problem (by referring to the deterministic analog of Theorem~\ref{ThmSLQ}). We remark that Equation (\ref{AuxMap}) is different from the deterministic counterpart of Equation (\ref{stocmax}). The mapping $z \mapsto \xi$ defined in this way is affine and maps $L^2_Q(0,T; \mathbb{R}^m)$ into itself. Our objective is to show that it admits a fixed point. To apply Banach Fixed Point Theorem, it suffices to show that the mapping $z \mapsto \xi$ is a contraction if $\mathbb{E}[x_0] = 0$. By considering the dynamics of $\xi_t^*\eta_t$, we have the following equality:
\begin{equation}\label{AuxId}
\begin{aligned}
&\quad \xi^*_TQ_T\xi_T + \int^T_0\xi^*_tQ_t\xi_t\,dt + \int^T_0\eta^*_tB_tR^{-1}_tB^*_t\eta_t\,dt\\
&= \int^T_0\eta^*_t\bar{A}_tz_t\,dt -\xi^*_T\bar{Q}_T(I - S_T)z_T - \int^T_0\xi^*_t\bar{Q}_t(I - S_t)z_t\,dt.
\end{aligned}
\end{equation}
Moreover, let $\phi(t;s)$ be the fundamental solution associated with $A_t$, and we get 
\begin{align*}
\eta_t &= \phi^*(T,t)(Q_T\xi_T+\bar{Q}_T(I - S_T)z_T)\\
&+ \int^T_t\phi^*(\tau,t)(Q_{\tau}\xi_{\tau}+\bar{Q}_{\tau}(I-S_{\tau})z_{\tau})\,d\tau.
\end{align*}
By the Cauchy-Schwarz inequality, we have 
\begin{align*}
\|\eta_t\| &\leq \interleave\phi\interleave_T\left(\left\|\xi\right\|_Q + \left\|Q^{-1/2}\bar{Q}(I-S)z\right\|\right)\\
&\leq \interleave\phi\interleave_T\left(\left\|\xi\right\|_Q + \interleave\mathcal{S}\interleave \|z\|_Q\right).
\end{align*}
Therefore, from (\ref{AuxId}), 
\[
\|\xi\|^2_Q \leq \sqrt{T}\interleave\phi\interleave_T(\|\xi\|_Q + \interleave\mathcal{S}\interleave_T \|z\|_Q) \interleave\bar{A}\interleave_T \|z\|_Q + \|\xi\|_Q \interleave\mathcal{S}\interleave_T \|z\|_Q,
\]
which shows that $z \mapsto \xi$ is a contraction if $\sqrt{T}\interleave\phi\interleave_T \interleave\bar{A}\interleave_T (1 + \interleave\mathcal{S}\interleave_T) + \interleave\mathcal{S}\interleave_T < 1$.
\end{proof}

\begin{cor}
If $S = I$, then $\mathcal{S} = 0$ and the previous condition reduces to \[\sqrt{T}\interleave\phi\interleave_T\interleave\bar{A}\interleave_T < 1.
\]
\end{cor}

\begin{rk}
For a single person optimization problem (i.e. the classical Linear-Quadratic Stochastic Control Problem), that is $\bar{A} = \mathcal{S} = 0$, we recover the standard existence and uniqueness result in the literature.
\end{rk}

\begin{rk}Assume that $\mathcal{S}_T = 0$. Then the non-singularity of $Q_T$ is not necessary and the norm $\interleave\mathcal{S}\interleave_T$ can be weaken to $\sup_{0 < t < T}\left\|Q^{-1/2}_t\bar{Q}_t (I - S_t)Q^{-1/2}_t\right\|$ in applying Theorem~\ref{MainThm}.
\end{rk}

\begin{rk}\label{PositiveCond}Suppose that $Q + \mathcal{S}$ can be written as $\mathcal{Q} + (Q + \mathcal{S} - \mathcal{Q})$, where $\mathcal{Q} > 0$ (that is, $\mathcal{Q}$ is positive definite) is chosen to satisfy suitable conditions stated in Theorem~\ref{MainThm}. Replacing $Q$, $\mathcal{S}$ by $\mathcal{Q}$ and $Q + \mathcal{S} - \mathcal{Q}$ respectively in the iterative scheme (\ref{AuxMap}) in the proof, a different sufficient condition for the unique existence of the equilibrium strategy is obtained:
\[
\sqrt{T}\interleave\phi\interleave_{\mathcal{Q},T} \interleave\bar{A}\interleave_{\mathcal{Q},T} (1 + \interleave\mathcal{S}\interleave_{\mathcal{Q},T}) + \interleave\mathcal{S}\interleave_{\mathcal{Q},T} < 1,
\]
where
\begin{align*}
\interleave\phi\interleave_{\mathcal{Q},T} &\triangleq \sup_{0 \leq t \leq T}\sqrt{\left\|\phi^*(T,t){\mathcal{Q}}^{1/2}_T\right\|^2 + \int^T_t\left\|\phi^*(s,t)\mathcal{Q}^{1/2}_s\right\|^2\,ds},\\
\interleave\bar{A}\interleave_{\mathcal{Q},T} &\triangleq \sup_{0 < t < T}\left\|\bar{A}_t\mathcal{Q}^{-1/2}_t\right\|,\\
\interleave\mathcal{S}\interleave_{\mathcal{Q},T} &\triangleq \sup_{0 < t \leq T}\left\|{\mathcal{Q}}^{-1/2}_t(Q_t + \mathcal{S}_t - \mathcal{Q}_t){\mathcal{Q}}^{-1/2}_t\right\|.
\end{align*}
For example, if all the coefficients are constants and $\bar{A} = 0$, then the condition $Q + \bar{Q}(I - S) > 0$ provides the desired unique existence by setting $\mathcal{Q} \triangleq Q + \bar{Q}(I-S)$ and $Q + \mathcal{S} - \mathcal{Q} \triangleq O$.
\end{rk}

\begin{rk}
In Appendix, an  example will be constructed which illustrates that its unique existence could be covered by our theory but it fails to satisfy the sufficient condition as stated in Huang et al.~\cite{HCM4}.
\end{rk}

\subsection{Relationship with Nonsymmetric Riccati Equation}
We can look for a solution of Equation (\ref{NewRic}) in the form $\bar{p}_t = \Gamma_t\bar{y}_t$. Hence we get the following Nonsymmteric Riccati Equation:
\begin{equation}\label{RicType}
\frac{d\Gamma_t}{dt} + \Gamma_t(A_t + \bar{A}_t) + A^*_t \Gamma_t - \Gamma_tB_tR^{-1}_tB^*_t\Gamma_t + Q_t +  \mathcal{S}_t = 0, \quad \Gamma_T = Q_T + \mathcal{S}_T.
\end{equation}
If it is solvable, using Remark~\ref{RicControl}, we have $\Gamma_t\bar{y}_t = \bar{p}_t = \Xi_t\bar{y}_t + \zeta_t$. Therefore, the optimal control $u$ is $-R^{-1}B^*(\Xi y + (\Gamma - \Xi)\bar{y})$ and the optimal trajectory $y$ satisfies
\begin{equation*}
\left\{
\begin{aligned}
dy_t &= [(A_t - B_tR^{-1}_tB^*_t\Xi_t)y_t + (\bar{A}_t - B_tR^{-1}_tB^*_t(\Gamma_t - \Xi_t))\bar{y}_t)]\,dt + \sigma_t\,dW_t,\\
y_0 &= x_0.
\end{aligned}
\right.
\end{equation*}
However, because of the non-zero term $\bar{A}_t$ and $\mathcal{S}_t$, Equation (\ref{RicType}) is not the standard Riccati Equation. Hence, it is not always solvable and no natural sufficient condition for the existence of the solution is known (see Freiling~\cite{Fre}). Moreover, $\Gamma_t$ is not necessarily symmetric. Nevertheless, when $n = 1$ and $Q + \mathcal{S} \geq 0$, the Nonsymmetric Riccati Equation becomes
\begin{equation*}
\left\{
\begin{aligned}
\frac{d\Gamma_t}{dt} &+ \Gamma_t\left(A_t + \frac{1}{2}\bar{A}_t\right) + \left(A_t + \frac{1}{2}\bar{A}_t\right)^*\Gamma_t - \Gamma_tB_tR^{-1}_tB^*_t\Gamma_t + Q_t + \mathcal{S}_t = 0,\\
\Gamma_T &= Q_T + \mathcal{S}_T,
\end{aligned}
\right.
\end{equation*}
which is of the standard form and the existence result holds. The explicit form of the solution $\Gamma_t$ can be established in this special case as follows. For the sake of simplicity, assume that all the coefficients are time-independent and our Riccati Equation can be simplified as:
\[
\frac{d\Gamma_t}{dt} + (2A + \bar{A})\Gamma_t  - B^2R^{-1}\Gamma^2_t + Q + \mathcal{S}= 0, \quad \Gamma_T = Q_T + \mathcal{S}_T.
\]
\begin{enumerate}
\item For $B = 0$, we have
\[
\Gamma_t = \left(Q_T + \mathcal{S}_T + \frac{Q + \mathcal{S}}{2A + \bar{A}}\right)\exp((2A + \bar{A})(T-t)) - \frac{Q + \mathcal{S}}{2A + \bar{A}},
\]
when $2A + \bar{A} \neq 0$ and
\[
\Gamma_t = (Q + \mathcal{S})(T-t) + Q_T + \mathcal{S}_T,
\]
when $2A + \bar{A} = 0$.
\item For $B \neq 0$, let $\alpha \geq 0 $ and $-\beta \leq 0$ be the two distinct roots of the quadratic equation
\[
Q + \mathcal{S} + (2A + \bar{A})\gamma  - B^2R^{-1}\gamma^2 = 0,
\]
and the solution can be explicitly written as
\[
\Gamma_t - \alpha = \frac{(Q_T + \mathcal{S}_T - \alpha)(\alpha + \beta)}{(Q_T + \mathcal{S}_T + \beta)\exp(B^2R^{-1}(\alpha + \beta)(T - t)) - (Q_T + \mathcal{S}_T - \alpha)},
\]
which is well-defined.
\end{enumerate}
We remark that for $n=1$, $Q + \mathcal{S}$ is not always nonnegative, and hence the Nonsymmetric Riccati Equation is not always solvable, the sufficient condition provided in Theorem~\ref{MainThm} may have to be invoked. The following proposition, Radon's lemma in Freiling~\cite{Fre}, or a time-dependent version of Theorem 4.3 on page 48 in Ma and Yong~\cite{MY}, justifies this claim.

\begin{prop}\label{SolvRic}Suppose that the following system of ordinary differential equations
\begin{equation*}
\left\{
\begin{aligned}
\frac{d}{dt}
\begin{pmatrix}
\xi_t \\ -\eta_t
\end{pmatrix}
&=
\begin{pmatrix}
A_t + \bar{A}_t  & -B_tR_t^{-1}B_t^*\\
Q_t  + \mathcal{S}_t& A^*_t
\end{pmatrix}
\begin{pmatrix}
\xi_t \\ \eta_t
\end{pmatrix},\\
\xi_{t_0} &= 0,\\
\eta_T &= (Q_T + \mathcal{S}_T)\xi_T,
\end{aligned}
\right.
\end{equation*}
admits a unique solution for any $t_0 \in [0, T]$. Then there is a unique solution $\Gamma_t$ of the Nonsymmetric Riccati Equation~(\ref{RicType}).
\end{prop}
\begin{proof}We first rewrite the system of ordinary differential equations as
\begin{equation*}
\left\{
\begin{aligned}
\frac{d}{dt}
\begin{pmatrix}
\xi_t \\ \eta_t
\end{pmatrix}
&=
\begin{pmatrix}
A_t + \bar{A}_t  & -B_tR_t^{-1}B_t^*\\
-Q_t  - \mathcal{S}_t& -A^*_t
\end{pmatrix}
\begin{pmatrix}
\xi_t \\ \eta_t
\end{pmatrix},\\
\xi_{t_0} &= 0,\\
\eta_T &= (Q_T + \mathcal{S}_T)\xi_T.
\end{aligned}
\right.
\end{equation*}
Let $\Phi(t, s)$ be the fundamental solution of this system of forward-backward ordinary differential equations. Then we have 
\begin{align*}
0 &= \begin{pmatrix}Q_T + \mathcal{S}_T,& -I\end{pmatrix}\begin{pmatrix}\xi_T\\ \eta_T\end{pmatrix}\\
&= \begin{pmatrix}Q_T + \mathcal{S}_T,& -I\end{pmatrix}\Phi(T,t_0)\begin{pmatrix}0\\ \eta_{t_0}\end{pmatrix}\\
&= \begin{pmatrix}Q_T + \mathcal{S}_T,& -I\end{pmatrix}\Phi(T,t_0)\begin{pmatrix}O\\ I\end{pmatrix}\eta_{t_0}.
\end{align*}
By the unique existence of the solution, the matrix $\begin{pmatrix}Q_T + \mathcal{S}_T,& -I\end{pmatrix}\Phi(T,t_0)\begin{pmatrix}O\\ I\end{pmatrix}$ is invertible for any $t_0 \in [0, T]$. By setting 
\[
\Gamma_t \triangleq -\left[\begin{pmatrix}Q_T + \mathcal{S}_T,& -I\end{pmatrix}\Phi(T,t)\begin{pmatrix}O\\ I\end{pmatrix}\right]^{-1}\left[\begin{pmatrix}Q_T + \mathcal{S}_T,& -I\end{pmatrix}\Phi(T,t)\begin{pmatrix}I\\ O\end{pmatrix}\right],
\]
it can be checked that it solves for the Nonsymmetric Riccati Equation~(\ref{RicType}).
\end{proof}

\begin{cor}Assume that $\interleave \bar{A} \interleave_{T_0} < \infty$ and $\interleave S \interleave_{T_0} < \infty$. The Nonsymmetric Riccati Equation~(\ref{RicType}) is solvable if either one of following conditions is satisfied:
\begin{enumerate}
\item $\interleave \bar{A} \interleave_{T_0} = 0$ and $ \interleave \mathcal{S} \interleave_{T_0} < 1$,
\item $\interleave \bar{A} \interleave_{T_0} \neq 0$ and
\[
T < \left(\frac{1 - \interleave \mathcal{S} \interleave_{T_0}}{\interleave \phi \interleave_{T_0} \interleave \bar{A} \interleave_{T_0} (1 + \interleave \mathcal{S} \interleave_{T_0})}\right)^2 \wedge T_0.
\]
\end{enumerate}
In particular, if all the coefficients are time-independent, then the Nonsymmetric Riccati Equation~(\ref{RicType}) is solvable if either one of following conditions is satisfied:
\begin{enumerate}
\item $\interleave \bar{A} \interleave = 0$ and $\interleave \mathcal{S} \interleave < 1$,
\item $\interleave \bar{A} \interleave \neq 0$ and
\[
T < \left(\frac{1 - \interleave \mathcal{S} \interleave}{\interleave \phi \interleave \interleave \bar{A} \interleave (1 + \interleave \mathcal{S} \interleave)}\right)^2.
\]
\end{enumerate}
\end{cor}

\subsection{The case that $n = 2$}
When $n = 2$, in general, the existence and uniqueness result of Equation (\ref{NewRic}) cannot be guaranteed. For the ease of computation, in this subsection, we will assume that all coefficients are constant matrices of suitable sizes and $S = I$. The latter assumption implies that the existence and uniqueness of the equilibrium do not depend on $\bar{Q}$ and the non-singularity of $Q_T$ is not required in Theorem~\ref{MainThm}. We denote the zero matrix by $O$.

\subsubsection{When $\bar{A} = -A$}
Let 
\[
A = \begin{pmatrix}-2.1 & -1.9\\ -1.2 & 1.7\end{pmatrix}, \quad R^{-1} = \begin{pmatrix}2 & 3.1\\ 3.1 & 4.9\end{pmatrix}, \quad Q = \begin{pmatrix}3.6 & -0.6\\ -0.6 & 0.2\end{pmatrix},
\]
$\bar{A} = -A$, $B = I$ and $Q_T = O$. Equation (\ref{NewRic}) becomes
\begin{equation*}
\frac{d}{dt}
\begin{pmatrix}
\xi_t \\ \eta_t
\end{pmatrix}
=\Pi\begin{pmatrix}
\xi_t \\ \eta_t
\end{pmatrix},
\end{equation*}
$\eta_T = 0$ and $\xi_0 = \mathbb{E}[x_0]$, where
\[
\Pi \triangleq \begin{pmatrix}
0 & 0 & 2 & 3.1\\
0 & 0 & 3.1 & 4.9\\
3.6 & -0.6 & 2.1 & 1.2\\
-0.6 & 0.2 & 1.9 & -1.7
\end{pmatrix}.
\]
We denote the fundamental solution of this system by 
\[
\Phi_t \triangleq \begin{pmatrix}\Phi^{11}_t & \Phi^{12}_t\\ \Phi^{21}_t & \Phi^{22}_t\end{pmatrix} = \exp(\Pi\cdot t).
\]
This equation is (uniquely) solvable if and only if there is a (unique) $\eta_0$ such that 
\begin{equation}\label{AlgEq}
0 = \begin{pmatrix}O& I\end{pmatrix}\begin{pmatrix}\Phi^{11}_T & \Phi^{12}_T\\ \Phi^{21}_T & \Phi^{22}_T\end{pmatrix}\begin{pmatrix}\mathbb{E}[x_0]\\ \eta_0\end{pmatrix} = \Phi^{21}_T \cdot \mathbb{E}[x_0] + \Phi^{22}_T \cdot \eta_0.
\end{equation}
Numerically, the determinant of $\Phi^{22}_{0.83}$ and $\Phi^{22}_{0.86}$ are approximately equal to $0.1244555$ and $-0.1295142$ respectively. Since the function 
\[
t \mapsto \det \left\{\begin{pmatrix}O & I \end{pmatrix}\exp(\Pi \cdot t)\begin{pmatrix}O \\ I \end{pmatrix}\right\} = \det(\Phi^{22}_t)
\]
is a continuous function, there is a scalar $T_0 \in (0.83, 0.86)$ such that $\det(\Phi^{22}_{T_0}) = 0$ and hence Equation (\ref{NewRic}) does not have any unique solution when $T = T_0$. In fact, the singularity of $\Phi^{22}_{T_0}$ shows that Equation (\ref{AlgEq}) cannot be solvable for any $\mathbb{E}[x_0]$ when $T = T_0$. Assume otherwise, that is, $\mathcal{R}(\Phi_T^{21}) \subseteq \mathcal{R}(\Phi_T^{22})$, where the range and the null space of a matrix $M$ are denoted by $\mathcal{R}(M)$ and $\mathcal{N}(M)$ respectively.  By the standard result in Linear Algebra, we have $\mathcal{N}((\Phi^{22}_T)^*) \subseteq \mathcal{N}((\Phi^{21}_T)^*)$. Choose $T = T_0$ as defined above and hence $\mathcal{N}((\Phi^{22}_{T_0})^*)$ is non-empty. It shows that 
\[
\text{rank}(\Phi_{T_0}) = \text{rank}\begin{pmatrix}\Phi^{11}_{T_0} & \Phi^{12}_{T_0}\\ \Phi^{21}_{T_0} & \Phi^{22}_{T_0}\end{pmatrix} = \text{rank}\begin{pmatrix}(\Phi^{11}_{T_0})^* & (\Phi^{21}_{T_0})^*\\ (\Phi^{12}_{T_0})^* & (\Phi^{22}_{T_0})^*\end{pmatrix} \leq 2n -1,
\]
which contradicts the invertibility of $\Phi_{T_0}$, and the claim holds.

\smallskip

For the general existence issue, it suffices to show that $\Phi^{21}_{T_0}$ is non-singular. In fact, we can choose $\mathbb{E}[x_0]$ such that $\Phi^{21}_{T_0} \cdot \mathbb{E}[x_0]$ is not in the range of $\Phi^{22}_{T_0}$ and hence Equation (\ref{AlgEq}) is not solvable. Figure 1 shows the relationship between the time variable $t$ and $D(t) \triangleq \det(\Phi^{21}_t)$ on the interval $[0,1]$ and demonstrates the non-singularity of $\Phi^{21}_{T_0}$.

\begin{figure}[ht]
\label{Figure 1} \centering
\includegraphics[height=3in]{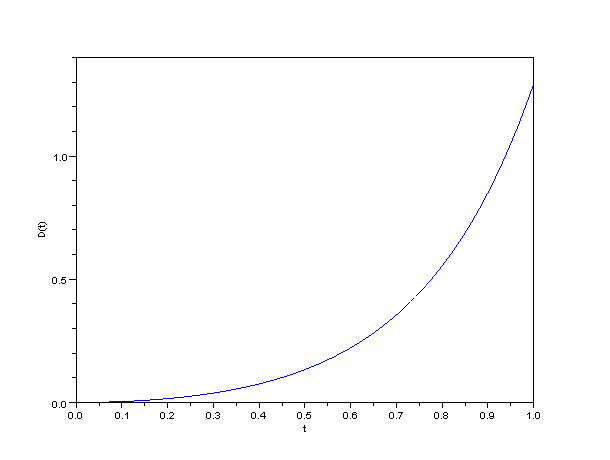}
\caption{The relationship between $t$ and $D(t)$}
\end{figure}

\subsubsection{An arbitrary case, when $\bar{A} \neq -A$}
Let 
\begin{alignat*}{2}
A &= \begin{pmatrix}-0.4 & -0.6\\ 0.4 & -0.1\end{pmatrix}, &\quad \bar{A} &= \begin{pmatrix}1.5 & 0.7\\ -0.1 & -0.3\end{pmatrix}\\ R^{-1} &= \begin{pmatrix}1.2 & 2.2\\ 2.2 & 4.4\end{pmatrix}, &\quad Q &= \begin{pmatrix}6.6 & -2.8\\ -2.8 & 1.2\end{pmatrix},
\end{alignat*}
$B = I$ and $Q_T = O$. Equation (\ref{NewRic}) becomes
\begin{equation*}
\frac{d}{dt}
\begin{pmatrix}
\xi_t \\ \eta_t
\end{pmatrix}
=\Pi\begin{pmatrix}
\xi_t \\ \eta_t
\end{pmatrix},
\end{equation*}
$\eta_T = 0$ and $\xi_0 = \mathbb{E}[x_0]$, where
\[
\Pi \triangleq \begin{pmatrix}
1.1 & 0.1 & 1.2 & 2.2\\
0.3 & -0.4 & 2.2 & 4.4\\
6.6 & -2.8 & 0.4 & -0.4\\
-2.8 & 1.2 & 0.6 & 0.1
\end{pmatrix}.
\end{equation*}
Numerically, the determinant of $\Phi^{22}_1$ is approximately equal to $-0.3582768$. Similar to the previous discussion, there is a scalar $T_0 > 0$ such that Equation (\ref{NewRic}) does not have any unique solution and there is no existence result for all initial points $\mathbb{E}[x_0]$.

\section{$\epsilon$-Nash Equilibrium}\label{EpsNash}

In this section, we shall show that the equilibrium strategy $u^i$ of Problem~\ref{MFG} is an $\epsilon$-Nash Equilibrium of Problem~\ref{FiniteDim}. For more inspiring elaboration on the notion of $\epsilon$-Nash Equilbrium, one can refer to, for example, Cardaliaguet~\cite{Car} and Huang et al.~\cite{HCM1,HCM5,HCM2}. Because of the permutation symmetry, it suffices to consider Player 1. In this section, $K$ will denote a generic constant, being independent of time, which may be different in line by line. In order to show that $(u^1. \ldots, u^N)$ is an $\epsilon$-Nash Equilibrium, it is crucial to prove that for any $\epsilon > 0$, there is a positive integer $N_0$ such that when $N \geq N_0$, we have 
\begin{equation}\label{EpsNashIneq}
\mathcal{J}^1(v^1, u^2, \ldots, u^N) \geq \mathcal{J}^1(u^1, \ldots, u^N) - \epsilon,
\end{equation}
for any admissible control $v^1$.

\smallskip

To begin with, we first approximate $y^i$, $1 \leq i \leq N$, where
\begin{equation*}
\left\{
\begin{aligned}
dy^i_t &= \left(A_ty^i_t + B_tu^i_t + \bar{A_t}\cdot\frac{1}{N-1}\sum_{j = 1, j \neq i}^Ny^j_t\right) dt + \sigma_t\,dW^i_t,\\
y^i_0 &= x_0^i.
\end{aligned}
\right.
\end{equation*}
\begin{prop}\label{McKean}As $N \to \infty$, we have
\[
\sup_{1 \leq i \leq N}\mathbb{E}\left[\sup_{0 \leq t \leq T}\|y_t^i - \hat{y}_t^i\|^2\right] = O\left(\frac{1}{N}\right).
\]
\end{prop}
\begin{proof}Recall that $\hat{y}$ is defined right after Problem~\ref{MFG}, we first note that 
\begin{equation*}
\left\{
\begin{aligned}
d(y^i_t - \hat{y}_t^i) &= \left(A_t(y^i_t - \hat{y}^i_t) + \bar{A_t}\cdot\frac{1}{N-1}\sum_{j = 1, j \neq i}^N(y^j_t - \mathbb{E}[\hat{y}^i_t])\right) dt,\\
y^i_0 - \hat{y}_0^i &= 0,
\end{aligned}
\right.
\end{equation*}
as $\hat{y}^1, \ldots, \hat{y}^N$ are i.i.d.. Taking square on both sides, we have
\begin{align*}
\|y^i_t - \hat{y}_t^i\|^2 \leq &\,\, K \int^t_0 \left(\|y^i_s - \hat{y}_s^i\|^2 + \frac{1}{(N - 1)^2} \left\|\sum_{j = 1, j\neq i}^N(y_s^j - \hat{y}_s^j)\right\|^2\right)\,ds\\
&+ K \int^t_0\frac{1}{(N - 1)^2} \left\|\sum_{j = 1, j\neq i}^N(\hat{y}_s^j - \mathbb{E}[\hat{y}^j_s])\right\|^2\,ds.
\end{align*}
By first applying Jensen's inequality to the second term in the first integrand and taking expectations on both sides, as $y^1 - \hat{y}^1, \ldots, y^N - \hat{y}^N$ are identically distributed and $\hat{y}^1, \ldots, \hat{y}^N$ are i.i.d., we have
\[
\mathbb{E}[\|y^i_t - \hat{y}_t^i\|^2] \leq K \int^t_0 \left(\mathbb{E}[\|y^i_s - \hat{y}_s^i\|^2] + \frac{1}{N - 1} \mathbb{E}[\|\hat{y}_s^i - \mathbb{E}[\hat{y}^i_s]\|^2]\right)\,ds.
\]
By the Gronwall's inequality, 
\[
\mathbb{E}[\|y^i_t - \hat{y}_t^i\|^2] \leq K \cdot \frac{1}{N - 1}\int^T_0\mathbb{E}[\|\hat{y}_s^i - \mathbb{E}[\hat{y}^i_s]\|^2]\,ds,
\]
as desired and note that the value of the last integral is independent of $i$.
\end{proof}
Note that the previous estimates are standard in the McKean-Vlasov Model, see for example Sznitman~\cite{Sz}. The following corollary is crucial in proving that $(u^1, \ldots, u^N)$ is an $\epsilon$-Nash Equilibrium.

\begin{cor}\label{NormSq}As $N \to \infty$, we have
\[
\mathbb{E}\left[\sup_{0 \leq t \leq T}|\|y_t^1\|^2 - \|\hat{y}_t^1\|^2|\right] = O\left(\frac{1}{\sqrt{N}}\right).
\]
\end{cor}
\begin{proof}
\begin{align*}
&\mathbb{E}\left[\sup_{0 \leq t \leq T}|\|y_t^1\|^2 - \|\hat{y}_t^1\|^2|\right]\\
\leq&\,\, \mathbb{E}\left[\sup_{0 \leq t \leq T}\|y_t^1 - \hat{y}_t^1\|^2\right] + 2\sqrt{\mathbb{E}\left[\sup_{0 \leq t \leq T}\|\hat{y}_t^1\|^2\right]}\sqrt{\mathbb{E}\left[\sup_{0 \leq t \leq T}\|y_t^1 - \hat{y}_t^1\|^2\right]},
\end{align*}
by the Cauchy-Schwarz Inequality and the result follows by applying Proposition~\ref{McKean}.
\end{proof}

Similarly, we have
\[
\mathbb{E}\left[\sup_{0 \leq t \leq T}\left|\left\|y_t^1 - S_t\cdot\frac{1}{N - 1}\sum_{j = 2}^Ny^j_t\right\|^2 - S_t \cdot \left\|\hat{y}_t^i - \frac{1}{N - 1}\sum_{j = 2}^N\hat{y}^j_t\right\|^2\right|\right] = O\left(\frac{1}{\sqrt{N}}\right), 
\]
as $\hat{y}^1, \ldots, \hat{y}^N$ are i.i.d.. Combining, we have 
\begin{equation}\label{RHS}
\mathcal{J}^1(\hat{v}^1, \ldots, \hat{v}^N) = J^1(\hat{v}^1) + O\left(\frac{1}{\sqrt{N}}\right).
\end{equation}
Because of the positive (non-negative) definiteness of the matrix-valued parameters, 
\[
\mathcal{J}^1(v^1, u^2, \ldots, u^N) \geq \mathbb{E}\left[\frac{1}{2}\int^T_0(v^1_t)^*R_tv^1_t\,dt\right] \geq \frac{\delta}{2}\,\mathbb{E}\left[\int^T_0 \|v^1_t\|^2\,dt\right], 
\]
it suffices to consider the admissible controls $v^1_t$ satisfying
\[
\mathbb{E}\left[\int^T_0 \|v^1_t\|^2\,dt\right] \leq \frac{2}{\delta}(J^1(u^1) + 1),
\]
otherwise the inequality (\ref{EpsNashIneq}) holds trivially.

\smallskip

Having the $L^2$-boundedness of the admissible controls, we are now ready to estimate $\mathcal{J}^1(v^1, u^2, \ldots, u^N)$. Recall that
\begin{equation*}
\left\{
\begin{aligned}
dx^1_t &= \left(A_tx^1_t + B_tv^1_t + \bar{A_t}\cdot\frac{1}{N-1}\sum_{j = 1, j \neq i}^Nx^j_t\right)\,dt + \sigma_t\,dW^1_t,\\
x^1_0 &= \xi^1,
\end{aligned}
\right.
\end{equation*} 
and for $i \neq 1$,
\begin{equation*}
\left\{
\begin{aligned}
dx^i_t &= \left(A_tx^i_t + B_tu^i_t + \bar{A_t}\cdot\frac{1}{N-1}\left(\sum_{j = 2, j \neq i}^Nx^j_t + x^1_t\right)\right)\,dt + \sigma_t\,dW^i_t,\\
x^i_0 &= \xi^i.
\end{aligned}
\right.
\end{equation*}
For $i \neq 1$, we claim that $x^i_t$ can be approximated by $\check{x}^i_t$, where
\begin{equation*}
\left\{
\begin{aligned}
d\check{x}^i_t &= \left(A_t\check{x}^i_t + B_tu^i_t + \bar{A_t}\cdot\frac{1}{N-1}\sum_{j = 2, j \neq i}^N\check{x}^j_t\right) dt + \sigma_t\,dW^i_t,\\
x^i_0 &= \xi^i.
\end{aligned}
\right.
\end{equation*}
The following proposition justifies this claim.

\begin{prop}As $N \to \infty$, we have 
\[
\sup_{2 \leq i \leq N}\mathbb{E}\left[\sup_{0 \leq t \leq T}\|x_t^i - \check{x}_t^i\|^2\right] = O\left(\frac{1}{N^2}\right).
\]
\end{prop}
\begin{proof}
For each $i \neq 1$, 
\begin{equation*}
\begin{aligned}
\|x^i_t\|^2 \leq &\,\, K\|\xi^i\|^2 + K\int ^t_0\left(\|x^i_s\|^2 + \|u^i_s\|^2 + \frac{1}{(N - 1)^2}\left\|\sum_{j = 1, j \neq i}^Nx^j_s\right\|^2\right)\,ds\\
&+ K\left\|\int^t_0\sigma_s\,dW^i_s\right\|^2,
\end{aligned}
\end{equation*}
and
\begin{equation*}
\begin{aligned}
\|x^1_t\|^2 \leq&\,\, K\|\xi^1\|^2 + K\int ^t_0\left(\|x^1_s\|^2 + \|v^1_s\|^2 + \frac{1}{(N - 1)^2}\left\|\sum_{j = 2}^Nx^j_s\right\|^2\right)\,ds\\
&+ K\left\|\int^t_0\sigma_s\,dW^1_s\right\|^2.
\end{aligned}
\end{equation*}
Taking the summation of $i$ from $1$ to $N$ on both sides and then using the same argument as in the proof of Proposition~\ref{McKean}, we have
\begin{equation*}
\begin{aligned}
\mathbb{E}\left[\sum_{i = 1}^N\|x^i_t\|^2\right] \leq&\,\, K\,\mathbb{E}\left[\sum_{i = 1}^N\|\xi^i\|^2\right]\\
&+K\int ^t_0\left(\mathbb{E}\left[\sum_{i = 1}^N\|x^i_s\|^2\right] + \mathbb{E}[\|v^1_s\|^2] + \left[\sum_{i = 2}^N\mathbb{E}[\|u^i_s\|^2]\right]\right)\,ds\\
&+ K\,\sum_{i = 1}^N\mathbb{E}\left[\left\|\int^t_0\sigma_s\,dW^i_s\right\|^2\right],
\end{aligned}
\end{equation*}
which shows that $\mathbb{E}\left[\sum_{i = 1}^N\|x^i_t\|^2\right] = O(N)$ uniformly for all $t$. Therefore, $\mathbb{E}\left[\sup_{0 \leq t \leq T}\|x^1_t\|^2\right]$ is bounded uniformly with respect to $N$.

\smallskip

Now, for $i \neq 1$,
\begin{equation*}
\left\{
\begin{aligned}
d(x^i_t - \check{x}^i_t) = &\,\, A_t(x^i_t - \check{x}^i_t)\,dt\\
&+\left(\bar{A_t}\cdot\frac{1}{N-1}\sum_{j = 2, j \neq i}^N(x^j_t -  \check{x}^j_t) + \bar{A_t}\cdot \frac{1}{N - 1}\cdot x^1_t\right) dt,\\
x^i_0 - \tilde{x}^i_0 =& \,\, 0.
\end{aligned}
\right.
\end{equation*}
Therefore,
\begin{align*}
\|x^i_t - \check{x}^i_t\|^2 \leq & \,\, K\int ^t_0\|x^i_s- \check{x}^i_s\|^2\,ds\\
&+ K\int^t_0\left(\frac{1}{(N - 1)^2}\left\|\sum_{j = 1, j \neq i}^N(x^j_s - \check{x}^j_s)\right\|^2 +  \frac{1}{(N - 1)^2}\left\|x^1_s\right\|^2\right)\,ds\\
\leq & \,\,K\int ^t_0\|x^i_s - \check{x}^i_s\|^2\,ds\\
&+ K\int ^t_0\left(\frac{1}{N - 1}\sum_{j = 1, j \neq i}^N\left\|x^j_s - \check{x}^j_s\right\|^2 +  \frac{1}{(N - 1)^2}\left\|x^1_s\right\|^2\right)\,ds.
\end{align*}
Taking summation over $i$ from $2$ to $N$, by applying Gronwall's inequality again, we have 
\[
\mathbb{E}\left[\sum_{i = 2}^N\|x^i_s - \check{x}^i_s\|^2\right] \leq K\cdot\frac{1}{N-1}\int^T_0\|x^1_s\|^2\,ds.
\]
As $x^i - \check{x}^i$, $i \neq 1$, are identically distributed, we have $\mathbb{E}[\|x^i_t - \check{x}^i_t\|^2] = O\left(\frac{1}{N^2}\right)$ uniformly for all $i$ and $t$, as desired.
\end{proof}

Similarly, we can approximate $x^1$ by $\check{x}^1$, where
\begin{equation*}
\begin{aligned}
d\check{x}^1_t &= \left(A_t\check{x}^1_t + B_tv^1_t + \bar{A_t}\cdot\frac{1}{N-1}\sum_{j = 2}^N\check{x}^j_t\right) dt + \sigma_t\,dW^1_t,\\
\check{x}^1_0 &= \xi^1,
\end{aligned}
\end{equation*}
in the sense that 
\[
\mathbb{E}\left[\sup_{0 \leq t \leq T}\|x_t^1 - \check{x}_t^1\|^2\right] = O\left(\frac{1}{N^2}\right).
\]
Having these approximation results, similar to Corollary~\ref{NormSq}, we have 
\begin{align*}
\mathbb{E}\left[\sup_{0 \leq t \leq T}\left|\|x_t^1\|^2 - \|\check{x}_t^1\|^2\right|\right] &= O\left(\frac{1}{N}\right),\\
\mathbb{E}\left[\sup_{0 \leq t \leq T}\left|\left\|x_t^1- S_t \cdot \frac{1}{N-1}\sum_{i=2}^Nx^j_t\right\|^2 - \left\|\check{x}_t^1- S_t \cdot \frac{1}{N-1}\sum_{i=2}^N\check{x}^j_t\right\|^2\right|\right] &= O\left(\frac{1}{N}\right).
\end{align*}

\smallskip

By using an essentially the same approach in the proof of Proposition~\ref{McKean}, we also have \[
\sup_{2 \leq i \leq n}\mathbb{E}\left[\sup_{0 \leq t \leq T}\|\check{x}_t^i - \hat{y}_t^i\|^2\right] = O\left(\frac{1}{N}\right).
\]
Moreover, we have 
\[
\mathbb{E}\left[\sup_{0 \leq t \leq T}\|\check{x}_t^1 - \hat{x}_t^1\|^2\right] = O\left(\frac{1}{N}\right).
\]
Indeed, 
\begin{equation*}
\left\{
\begin{aligned}
d(\check{x}^1_t - \hat{x}^1_t) &= \left(A_t(\check{x}^1_t - \hat{x}^1_t) + \bar{A}_t\cdot\frac{1}{N - 1}\sum_{i = 2}^N(\check{x}_i - \mathbb{E}[\hat{y}^i_t])\right)dt,\\
\check{x}^1_0 - \hat{x}^1_0 &= 0,
\end{aligned}
\right.
\end{equation*}
and hence
\begin{equation*}
\begin{aligned}
\|\check{x}^1_t - \hat{x}^1_t\|^2 \leq & \,\, K\int^t_0\left(\|\check{x}^1_s - \hat{x}^1_s\|^2 + \frac{1}{N-1}\sum_{i = 2}^N\|\check{x}^j_s - \hat{y}^i_s\|^2\right)\,ds\\
&+ K\int^t_0\frac{1}{(N - 1)^2}\left\|\sum_{i = 2}^N(\hat{y}^i_s - \mathbb{E}[\hat{y}^i_s])\right\|^2\,ds.
\end{aligned}
\end{equation*}
By applying Gronwall's inequality, the claim can be deduced. In conclusion, by combining all these estimates, the main claim in this section follows.
\begin{thm}
$(u^1, \ldots, u^N)$ is an $\epsilon$-Nash Equilibrium of Problem~\ref{FiniteDim}.
\end{thm}
\begin{proof}
First recall that we have proved
\[
\mathcal{J}^1(u^1, \ldots, u^N) = J^1(u^1) + O\left(\frac{1}{\sqrt{N}}\right)
\]
in Equation (\ref{RHS}). Based on the above estimates, we have
\begin{eqnarray*}
\mathbb{E}\left[\sup_{0 \leq t \leq T}\|x_t^1 - \hat{x}_t^1\|^2\right] &=& O\left(\frac{1}{N}\right),\\
\mathbb{E}\left[\sup_{0 \leq t \leq T}\|x_t^i - \hat{y}_t^i\|^2\right] &=& O\left(\frac{1}{N}\right), 
\end{eqnarray*}
for $i \neq 1$. Using a similar argument as in the proof of Equation (\ref{RHS}), we have
\begin{align*}
\mathcal{J}^1(v^1, u^2, \ldots, u^N) &= J^1(v^1) + O\left(\frac{1}{\sqrt{N}}\right)\\
&\geq J^1(u^1) + O\left(\frac{1}{\sqrt{N}}\right)\\
&= \mathcal{J}^1(u^1, \ldots, u^N) + O\left(\frac{1}{\sqrt{N}}\right),
\end{align*}
as required.
\end{proof}

\section{Mean Field Type Linear-Quadratic Stochastic Control Problems}\label{MFTSC}

In this section, we will use the technique employed in the previous section to find the optimal control of the Mean Field Type counterpart of Problem~\ref{MFG}.

\begin{problem}\label{MFT} Let $\xi$ be a random vector that is identically distributed to $\xi^1$ and independent to $W$. The objective is to find an optimal control $u$ which minimizes
\begin{align*}
J(v) \triangleq&\,\, \mathbb{E}\left[\frac{1}{2}\int^T_0 x_t^* Q_t x_t + v_t^* R_tv_t  + (x_t - S_t\,\mathbb{E}[x_t])^* \bar{Q}_t(x_t - S_t\,\mathbb{E}[x_t]) \,dt \right]\\
&+ \mathbb{E}\left[\frac{1}{2} x_T^* Q_Tx_T + \frac{1}{2}(x_T - S_T\,\mathbb{E}[x_T])^* \bar{Q}_T(x_T - S_T\,\mathbb{E}[x_T])\right],
\end{align*}
where the dynamics is given by 
\[
dx_t = \left(A_tx_t + B_tv_t + \bar{A_t}\,\mathbb{E}[x_t]\right)\,dt + \sigma_t\,dW_t, \quad x(0) = x_0,
\]
and $v$ is a control in $L^2_{\mathcal{F}}(0, T; \mathbb{R}^m)$. 
\end{problem}

\smallskip

In order to solve this optimization problem, we still adopt the adjoint equation approach.
\begin{thm}Problem~\ref{MFT} is uniquely solvable if there is a unique solution $(y, p)$ of the following linear mean-field FBSDE:
\begin{equation}\label{MFTe}
\left\{
\begin{aligned}
d
\begin{pmatrix}
y_t \\ -p_t
\end{pmatrix}
=&
\begin{pmatrix}
A_t & -B_tR_t^{-1}B_t^*\\
Q_t+\bar{Q}_t & A_t^*
\end{pmatrix}
\begin{pmatrix}
y_t \\ p_t
\end{pmatrix}\,dt\\
&+
\begin{pmatrix}
\bar{A}_t\,\mathbb{E}[y_t] \\ -\bar{Q}_tS_t\,\mathbb{E}[y_t] - S^*_t\bar{Q}_t(I - S_t)\,\mathbb{E}[y_t] + {\bar{A}_t}^*\,\mathbb{E}[p_t]
\end{pmatrix}\,dt
+
\begin{pmatrix}
\sigma_t \\ q_t
\end{pmatrix}\,dW_t,\nonumber\\
y_0 =&\,\, x_0\\
p_T =&\,\, (Q_T+\bar{Q}_T)y_T - \bar{Q}_TS_T\,\mathbb{E}[y_T] - S^*_T\bar{Q}_T(I - S_T)\,\mathbb{E}[y_T].
\end{aligned}
\right.
\end{equation}
In this case, the optimal control is given by $u = -R^{-1}B^*p$.
\end{thm}
\begin{proof}
It is an immediate consequence of Theorem~4.1 in Andersson and Djehiche \cite{AD}. Note that the assumption of non-negativity is not needed because of the linearity of the mean field term. See also our Theorem~\ref{ThmSLQ}.
\end{proof}

By Theorem~\ref{ThmSLQ}, Equation (\ref{MFTe}) is uniquely solvable if and only if there is an unique pair $(\bar{y}, \bar{p})$ solving the following system of ordinary differential equations:
\begin{equation*}
\left\{
\begin{aligned}
\frac{d}{dt}
\begin{pmatrix}
\bar{y}_t \\ -\bar{p}_t
\end{pmatrix}
&=
\begin{pmatrix}
A_t+\bar{A}_t & -B_tR_t^{-1}B_t^*\\
Q_t + (I - S_t)^*\bar{Q}_t(I - S_t) & A_t^* + \bar{A}_t^*
\end{pmatrix}
\begin{pmatrix}
\bar{y}_t\\ \bar{p}_t
\end{pmatrix},\\
\bar{y}_0 &= \bar{x_0},\\
\bar{p}_T &= (Q_T + (I - S_T)^*\bar{Q}_T(I - S_T))\bar{y}_T,
\end{aligned}
\right.
\end{equation*}
where $\bar{x}_0$ is defined to be $\mathbb{E}[x_0]$. Unlike to Equation (\ref{NewRic}), it is standard in Optimal Control Theory from the nonnegative definiteness of $(Q_T + (I - S_T)^*\bar{Q}_T(I - S_T))$ and hence there is a unique solution. Therefore, the Mean Field Type Linear-Quadratic Stochastic Control Problem is uniquely solvable.

\section{Comparison of Problem~\ref{MFG} and Problem~\ref{MFT}}\label{Compare}
We will now compare the equilibrium strategy of Mean Field Game and the optimal control of Mean Field Type Stochastic Control Problem when $n = 1$ and $S = I$. We assume that all the coefficients are constant.
Since the equilibrium strategy and the optimal control are of the same form, they are different if we can show that $\psi_1(T) \neq \psi_2(T)$,
where $(\varphi_1, \psi_1)$ satisfies
\begin{equation*}
\left\{
\begin{aligned}
\frac{d}{dt}
\begin{pmatrix}
\varphi_1 \\ -\psi_1
\end{pmatrix}
&=
\begin{pmatrix}
A+\bar{A} & -BR^{-1}B^*\\
Q & A^*
\end{pmatrix}
\begin{pmatrix}
\varphi_1 \\ \psi_1
\end{pmatrix},\\
\varphi_1(0) &= \bar{x}_0,\\
\psi_1(T) &= Q_T\varphi_1(T),
\end{aligned}
\right.
\end{equation*}
and $(\varphi_2, \psi_2)$ satisfies
\begin{equation*}
\left\{
\begin{aligned}
\frac{d}{dt}
\begin{pmatrix}
\varphi_2 \\ -\psi_2
\end{pmatrix}
&=
\begin{pmatrix}
A+\bar{A} & -BR^{-1}B^*\\
Q & A^* + \bar{A}^*
\end{pmatrix}
\begin{pmatrix}
\varphi_2 \\ \psi_2
\end{pmatrix},\\
\varphi_2(0) &= \bar{x}_0\\
\psi_2(T) &= Q_T\varphi_2(T).
\end{aligned}
\right.
\end{equation*}

\smallskip

To illustrate their differences, we suppose that $Q = \bar{Q} = 0$, $\bar{A} \neq 0$, $R = 1$, $Q_T = 1$ and $B \neq 0$. Solving the equations, we have $\psi_1(t) = \psi_1(0) e^{-At}$ and $\psi_2(t) = \psi_2(0) e^{-(A + \bar{A})t}$, and the condition that $\psi_1(T) \neq \psi_2(T)$ is reduced to $\varphi_1(T) \neq \varphi_2(T)$.

\smallskip

Solving 
\[
\frac{d\varphi_1}{dt} = (A + \bar{A})\varphi_1 - B^2 \varphi_1(T) e^{A(T-t)},
\]
we have 
\[
e^{-(A + \bar{A})T}\varphi_1(T) - \bar{x}_0 = - B^2 \varphi_1(T) e^{AT} \frac{1}{2A + \bar{A}}(1 - e^{-(2A + \bar{A})T}).
\]
In a similar fashion, solving 
\[
\frac{d\varphi_2}{dt} = (A + \bar{A})\varphi_2 - B^2 \varphi_2(T) e^{(A+\bar{A})(T-t)},
\]
we have 
\[
e^{-(A + \bar{A})T}\varphi_2(T) - \bar{x}_0 = - B^2 \varphi_2(T) e^{(A+ \bar{A})T} \frac{1}{2A + 2\bar{A}}(1 - e^{-(2A + 2\bar{A})T}).
\]
Therefore, the condition that $\varphi_1(T) \neq \varphi_2(T)$ is reduced to 
\[
\frac{1}{2A + \bar{A}}(1 - e^{-(2A + \bar{A})T}) \neq e^{\bar{A}T} \frac{1}{2A + 2\bar{A}}(1 - e^{-(2A + 2\bar{A})T}).
\]
It can be seen that this condition holds if we choose $A$ sufficiently large and we conclude that the equilibrium strategy is in general different from the optimal control.

\section{Concluding Remarks}
In summary, we provided a comprehensive study of the unique existence of equilibrium strategies of LQMFGs by adopting the adjoint equation approach. In virtue of the linear structure of the adjoint equations, the optimal mean field term satisfies the forward-backward ordinary differential equation (\ref{NewRic}) in which, unlike the classical Riccati equation approach, the argument could be much easier to be extended in the higher dimensional settings. It is remarked that most existing literature, such as Huang et al.~\cite{HCM4}, only considered a particular example of an one dimensional problem in our LQMFGs in which they even failed to provide a complete solution except under a restrictive technical condition that excludes a large class of classical Linear Quadratic Stochastic Control problems. For the one dimensional case, we showed that the equilibrium strategy always exists uniquely; while for dimension greater than one, a sufficient condition for the unique existence of the equilibrium strategy is provided, which is independent of both the solutions of any Riccati equations and the coefficients of controls and is always satisfied whenever those of the mean-field term are vanished (and therefore including the classical LQSC problems as special cases). Finally, we also illustrate the fundamental differences between Linear-Quadratic Mean Field Type Stochastic Control Problems and  MFGs. 

\section{Acknowledgments}
We are grateful to many seminar and conference participants such as those in the workshop of Sino-French Summer Institute 2011 and 15th International Congress on Insurance Mathematics and Economics 2011 for their valuable comments and suggestions on the preliminary version of the present work. The first author acknowledges the financial support from WCU(World Class University) program through the National Research Foundation of Korea funded by the Ministry of Education, Science and Technology (R31 - 20007) and The Hong Kong RGC GRF 500111. The third author-Phillip Yam acknowledges the financial support from The Hong Kong RGC GRF 502909, The Hong Kong RGC GRF 500111, The Hong Kong RGC GRF 404012 with the project title: Advanced Topics In Multivariate Risk Management In Finance And Insurance, The Hong Kong Polytechnic University Collaborative Research Grant G-YH96, The Chinese University of Hong Kong Direct Grant 2010/2011 Project ID: 2060422, and The Chinese University of Hong Kong Direct Grant 2010/2011 Project ID: 2060444. Phillip Yam also expresses his sincere gratitude to the hospitality of both Hausdorff Center for Mathematics of the University of Bonn and Mathematisches Forschungsinstitut Oberwolfach (MFO) in the German Black Forest during the preparation of the present work. Finally, the forth author-S. P. Yung acknowledges the financial support from an HKU internal grants of code 200807176228 and 200907176207. 

\appendix

\section{Appendix: Comparison of the Approaches of Huang et al.~\cite{HCM4} and the Present Paper}
We consider the single agent problem studied by Huang et al.~\cite{HCM4} (HCM), that is, the distribution $F(a)$ in HCM is now a Dirac distribution, to compare the approaches of HCM and the present paper (BSYY) on the same problem. More precisely, we want to find the optimal control $u$ which minimizes the cost functional \begin{equation}\label{cPayoff}
J(u) = \mathbb{E}\left[\int^T_0|z_t - \gamma(\bar{z}_t+\eta)|^2 + ru_t^2\,dt\right],
\end{equation}
where
\begin{equation}\label{cDynamic}
dz_t = (az_t + bu_t)\,dt + \alpha\bar{z}_t\,dt + \sigma\,dw_t, \quad z(0) = z_0,
\end{equation}
$\bar{z}_t$ is fixed, deterministic and $z_0$ is a random variable with zero mean, independent of the Wiener process. At the second stage, we consider a fixed point problem
\begin{equation}\label{cFixProb}
\bar{z}_t = \mathbb{E}[z_t],
\end{equation}
where $z_t$ is the optimal state of the control problem (\ref{cPayoff}), (\ref{cDynamic}). In the sequel, we shall describe both approaches in detail to make the comparison explicit. As in HCM, we use the notation $z_t^* \triangleq \gamma(\bar{z}_t+\eta)$. 

\subsection{HCM approach}

For given $\bar{z}$, one solves the stochastic control problem (\ref{cPayoff}), (\ref{cDynamic}) by the Riccati differential equation approach. The optimal control is given by
\begin{equation}\label{cControlHCM}
u_t = -\frac{b}{r}(\Pi_t z_t + s_t),
\end{equation}
where $\Pi_t$ is the positive solution of the Riccati equation
\begin{equation}\label{cRic}
\frac{d\Pi_t}{dt} + 2a\Pi_t - \frac{b^2}{r}\Pi_t^2 + 1 = 0, \quad \Pi_T = 0,
\end{equation}
and $s_t$ solves the linear differential equation 
\begin{equation}\label{cODE}
\frac{ds_t}{dt} + \left(a - \frac{b^2}{r}\Pi_t\right)s_t + \alpha \Pi_t \bar{z}_t - z_t^* = 0, \quad s_T = 0.
\end{equation}
Therefore, from (\ref{cDynamic}), the optimal trajectory satisfies
\begin{equation*}\label{cTraj}
dz_t = \left(a - \frac{b^2}{r}\Pi_t\right)z_t\,dt + \left(-\frac{b^2}{r}s_t + \alpha\bar{z}_t\right)\,dt + \sigma\,dw_t, \quad z(0) = z_0.
\end{equation*}
Furthermore, to satisfy (\ref{cFixProb}), we must have
\begin{equation}\label{cExp}
\frac{d\bar{z}_t}{dt} = \left(a - \frac{b^2}{r}\Pi_t\right)\bar{z}_t - \frac{b^2}{r}s_t + \alpha\bar{z}_t, \quad \bar{z}(0) = 0,
\end{equation}
and it suffices to find a solution of the deterministic system (\ref{cODE}), (\ref{cExp}). We now state the sufficient condition in HCM that guarantees the unique solution for the system (\ref{cODE}), (\ref{cExp}). 

\subsection{HCM proof}
Introduce
\[
\Phi(t, \tau) = \exp\left(-\int^t_\tau\left(a - \frac{b^2}{r}\Pi_{\sigma}\right)d\sigma\right),
\]
and note that $t$ is not necessarily greater than $\tau$. Solving (\ref{cODE}) by the formula
\[
s_t = \int^T_t\Phi(t,\tau)((\alpha \Pi_{\tau} - \gamma)\bar{z}_{\tau} - \gamma \eta)\,d\tau,
\]
and substituting in (\ref{cExp}) yields
\begin{equation}\label{cFixHCM}
\bar{z}_t = \int^t_0 \Phi(\sigma,t)\left\{\alpha \bar{z}_{\sigma} - \frac{b^2}{r}\int^T_{\sigma}\Phi(\sigma,\tau)(\alpha\Pi_{\tau} - \gamma)\bar{z}_{\tau} - \gamma \eta\,d\tau\right\}\,d\sigma.
\end{equation}
\begin{rk}
There is a typo in HCM, on page 171, where the integral sign $\int^T_{\sigma}$ is written as $\int^T_0$.
\end{rk}
The problem is reduced to find a fixed point of equation (\ref{cFixHCM}) and HCM uses the contraction principle to solve for it in the space $C(0,T)$. 

\smallskip

Consider a continuous function $\varphi$ and the linear map 
\[
\Gamma \varphi(t) = \int^t_0 \Phi(\sigma, t)\left\{\alpha\varphi_{\sigma} - \frac{b^2}{r} \int^T_{\sigma}\Phi(\sigma,\tau)(\alpha \Pi_{\tau} - \gamma) \varphi_{\tau}\,d\tau\right\}\,d\sigma,
\]
then the norm $\|\Gamma\|$ must be required to be strictly less than 1. Since $\Phi$ and $\Pi$ are positive, we have
\[
|\Gamma\varphi(t)| \leq \|\varphi\|\int^t_0 \Phi(\sigma, t)\left\{|\alpha| + \frac{b^2}{r} \int^T_{\sigma}\Phi(\sigma,\tau)(|\alpha| \Pi_{\tau} + |\gamma|)\,d\tau\right\}\,d\sigma, 
\]
and thus the assumption
\begin{equation}\label{cContHCM}
\|\Gamma\| \leq \sup_{0 < t < T}\int^t_0 \Phi(\sigma, t)\left\{|\alpha| + \frac{b^2}{r} \int^T_{\sigma}\Phi(\sigma,\tau)(|\alpha| \Pi_{\tau} + |\gamma|)\,d\tau\right\}\,d\sigma < 1,
\end{equation}
guarantees the contraction property. 

\begin{rk}
There are two typos in the statement of the condition in HCM, on page 170 (although the previous one is corrected): $\frac{b^2}{r}$ appears as a multiplicative factor for the whole right hand side of (\ref{cContHCM}), which should not be. Also the integral sign $\int^t_0$ is written as $\int^T_0$.
\end{rk}
Correcting these typos, this is the HCM result in the present framework.

\subsection{BSYY approach}
The present paper uses the stochastic maximum principle to solve (\ref{cPayoff}), (\ref{cDynamic}). The optimal control is
\begin{equation}\label{cControlBSYY}
u_t = -\frac{b}{r}\,p_t,
\end{equation}
where $p_t = \mathbb{E}[\omega_t|\mathcal{F}_t]$, $\mathcal{F}_t$ is the filtration generated by $z_0$ and the Wiener process up to time $t$ and $\omega_t$ solves the adjoint equation
\[
-\frac{d\omega_t}{dt} = a \omega_t + z_t - z^*_t, \quad \omega_T = 0.
\]
Combining, the following system of necessary conditions holds:
\begin{equation}\label{cNesBSYY}
\left\{
\begin{aligned}
dz_t &= \left(az_t - \frac{b^2}{r}\,p_t + \alpha\bar{z}_t\right)\,dt + \sigma\,dw_t,\\
z(0) &= z_0,\\
-\frac{d\omega_t}{dt} &= a\omega_t + z_t - z^*_t,\\
\omega_T &= 0,
\end{aligned}
\right.
\end{equation}
where $p_t = \mathbb{E}[\omega_t | \mathcal{F}_t]$. It is well-known that $p_t = \Pi_tz_t + s_t$ and (\ref{cControlHCM}), (\ref{cControlBSYY}) coincide. The difference concerns the fixed point argument.

\smallskip

In our approach, we define $\bar{z}_t = \mathbb{E}[z_t]$ and $\bar{p}_t = \mathbb{E}[p_t] = \mathbb{E}[\omega_t]$. Therefore (\ref{cNesBSYY}) becomes 
\begin{equation}\label{cExpBSYY}
\left\{
\begin{aligned}
\frac{d\bar{z}_t}{dt} &= (a + \alpha)\bar{z}_t - \frac{b^2}{r}\,\bar{p}_t ,\\
\bar{z}(0) &= 0,\\
-\frac{d\bar{p}_t}{dt} &= a\bar{p}_t + (1 - \gamma)\bar{z}_t - \gamma \eta,\\
\bar{p}_T &= 0,
\end{aligned}
\right.
\end{equation}
which we shall solve. Again, $\bar{p}_t = \Pi_t\bar{z}_t + s_t$, and the systems (\ref{cExpBSYY}) and (\ref{cODE}), (\ref{cExp}) are equivalent. 

\subsection{BSYY proof and Nonsymmetric Riccati Equation}
The trouble with the relation $\bar{p}_t = \Pi_t \bar{z}_t + s_t$ is that it does not express the adjoint variable $\bar{p}_t$ as an affine function of $\bar{z}_t$ alone. In fact, we need both $\bar{z}_t$ and $s_t$, which is a coupled system, to be solved first, as done in HCM. Our approach is to express $\bar{p}_t$ as an affine function on $\bar{z}_t$ only. We write
\begin{equation}\label{cSepRic}
\bar{p}_t = P_t\bar{z}_t + \rho_t
\end{equation}
and by identification, we obtain:
\begin{align}\label{cRicType}
\frac{dP_t}{dt} &= -(2a + \alpha)P_t + \frac{b^2}{r}P^2_t - 1 + \gamma, \quad P_T = 0,\\
\frac{d\rho_t}{dt} &= -\left(a - \frac{b^2}{r}\right)\rho_t + \gamma \eta, \quad \rho_T = 0. \nonumber
\end{align}
The Riccati equation (\ref{cRicType}) is different from (\ref{cRic}) and it is called a nonsymmetric Riccati equation because in dimension larger than 1, it leads to (non-standard) nonsymmetric Riccati equatons. Solving (\ref{cExpBSYY}) amounts to solve the Riccati equation (\ref{cRicType}), since using (\ref{cSepRic}) in (\ref{cExpBSYY}), $\bar{z}_t$ is a solution of linear equation. 

\smallskip

If we assume that
\begin{equation}\label{cContBSYY}
\gamma \leq 1
\end{equation}
(which is independent of the choice of $b$), then the second order equation 
\[
-\frac{b^2}{r}\varsigma^2 + (2a + \alpha)\varsigma + 1 - \gamma = 0
\]
has two roots $\varsigma_1 \geq 0$, $\varsigma_2  \leq 0$ and the solution of the Riccati equation is 
\[
P_t = \frac{(1-\gamma)r}{b^2}\frac{\exp\left((\varsigma_1 - \varsigma_2)\frac{b^2}{r}(T-t)\right) - 1}{\varsigma_1-\varsigma_2\exp\left((\varsigma_1 - \varsigma_2)\frac{b^2}{r}(T-t)\right)}.
\]

\smallskip

We can compare assumption (\ref{cContBSYY}) with respect to assumption (\ref{cContHCM}), in obtaining the fixed point property. For instance, take $a = 0$, $\alpha = 0$ and $r = 1$, Condition (\ref{cContHCM}) means that
\begin{equation}\label{cContMod}
\sup_{0 < t < T}b^2|\gamma|\int^t_0\Phi(\sigma,t)\left(\int^T_{\sigma}\Phi(\sigma,\tau)\,d\tau\right)\,d\sigma < 1
\end{equation}
where
\[
\Phi(t, \tau) = \exp\left(b^2\int^t_{\tau}\Pi_{\sigma}\,d\sigma\right)
\]
and
\[
\frac{d\Pi_t}{dt} - b^2\Pi_t^2 + 1= 0, \quad \Pi_T = 0.
\]
Thus, \[
b\Pi_t + 1 = \frac{2}{1 + e^{-2b(T-t)}},
\]
and from (\ref{cContMod}),
\[
\sup_{0 < t < T}b^2|\gamma|\int^t_0\exp\left(\int^{\sigma}_tb^2 \Pi_{\lambda}\,d\lambda\right)\left\{\int^T_{\sigma}\exp\left(\int^{\sigma}_{\tau}b^2\Pi_{\mu}\,d\mu\right)\,d\tau\right\}\,d\sigma < 1,
\]
which means that
\begin{equation}\label{SimCond}
\sup_{0 < t < T}b^2|\gamma|\int^t_0\exp\left(-\int^t_{\sigma}b^2 \Pi_{\lambda}\,d\lambda\right)\left\{\int^T_{\sigma}\exp\left(-\int_{\sigma}^{\tau}b^2\Pi_{\mu}\,d\mu\right)\,d\tau\right\}\,d\sigma < 1.
\end{equation}
Since
\[
b\Pi_t = \frac{1 - e^{-2b(T-t)}}{1 + e^{-2b(T-t)}} < 1,
\]
from (\ref{SimCond}), we obtain
\begin{align*}
1&> \sup_{0 < t < T}b^2|\gamma|\int^t_0e^{-b(t-\sigma)}\left\{\int^T_{\sigma}e^{-b(\tau - \sigma)}\,d\tau\right\}\,d\sigma\\
&= \sup_{0 < t < T}b|\gamma|\int^t_0e^{-b(t-\sigma)}(1 - e^{-b(T-\sigma)})\,d\sigma\\
&= |\gamma|\sup_{0 < t < T}\left(1 - e^{-bt}-\frac{1}{2}e^{-b(T-t)}+\frac{1}{2}e^{-b(T+t)}\right),
\end{align*}
which amounts to 
\begin{equation}\label{cSimCond}
|\gamma|\,(1 - e^{-bT}) < 1.
\end{equation}
Obviously, (\ref{cContBSYY}) and (\ref{cSimCond}) are not equivalent. We remark that in order to guarantee the existence uniformly for arbitrarily choice of $b$, (\ref{cSimCond}) only holds when $|\gamma| \leq 1$. 

\subsection{Generalization to $n$ dimensions}

Both approaches can be considered in $n$ dimension. However, the condition for the existence and uniqueness of the fixed point in the HCM approach becomes extremely difficult to be checked, since it involves the solution of Riccati equations. Our approach leads to conditions which are much easier to be verified, and also introduces an interesting and new direction for solvable nonsymmetric Riccati equations that do not correspond to any usual control problems at all. The contribution of BSYY provides a complement to HCM theory, with insights which deserve to be known.


\begin{thebibliography}{99}

\bibitem{ACCD} \textsc{Achdou, Y., Camilli, F., Capuzzo-Dolcetta, I.} (2010). Mean field games: numerical methods for the planning problem. \textit{Preprint}.

\bibitem{AchCD} \textsc{Achdou, Y., Capuzzo-Dolcetta, I.} (2010). Mean field games: Numerical methods. \textit{SIAM
Journal on Numerical Analysis} \textbf{48} (3), 1136–1162.

\bibitem{AD} \textsc{Andersson, D., Djehiche, B.} (2010). A Maximum Principle for SDEs of Mean-Field Type. \textit{Applied Mathematics and Optimization} \textbf{63}(3), 341-356.

\bibitem{BP} \textsc{Bardi, M.} (2011). Explicit solutions of some Linear-Quadratic Mean Field Games. \textit{Preprint}.

\bibitem{B} \textsc{Bensoussan, A.} (1992). \textit{Stochastic Control of Partially Observable Systems}, Cambridge University Press, Cambridge.

\bibitem{BF1} \textsc{Bensoussan, A., Frehse, J.} (2000). Stochastic Games for N players. \textit{Journal of Optimization Theory and Applications}, \textbf{105} (3), 543-565.

\bibitem{BF2} \textsc{Bensoussan, A., Frehse, J.} (2009). On diagonal elliptic and parabolic systems with super-quadratic Hamiltonians. \textit{Communications on Pure and Applied Analysis}, \textbf{8}, 83-94.

\bibitem{BFV} \textsc{Bensoussan, A., Frehse, J., Vogelgesang, J.} (2010). Systems of Bellman equations to stochastic differential games with non-compact coupling. \textit{Discrete and Continuous Dynamical Systems}, \textit{27} (4), 1375-1389.

\bibitem{BDLP} \textsc{Buckdahn, R., Djehiche, B., Li, J., Peng, S.} (2009). Mean-field backward stochastic differential equations: A limit approach. \textit{The Annals of Probability}, \textit{37} (4), 1524-1565.

\bibitem{BLP} \textsc{Buckdahn, R., Li, J., Peng, S.} (2007). Mean-field backward stochastic differential equations and related partial differential equations. \textit{Stochastic Processes and their Applications} \textbf{119} (10), 3133-3154.

\bibitem{Car} \textsc{Cardaliaguet, P.} (2010). \textit{Notes on mean field games}.

\bibitem{FlSo} \textsc{Fleming, W.H., Souganidis, P E.} (1989). On the existence of value functions of two player, zero sum stochastic differential games. \textit{Indiana University Mathematics Journal} \textbf{38}, 293-314.

\bibitem{Fre} \textsc{Freiling, G.} (2002). A Survey of Nonsymmetric Riccati Equations. \textit{Linear Algebra and Its Applications} \textbf{351}, 243-270.

\bibitem{GMS} \textsc{Gomes, D.A., Mohr, J., Souza, R. R.} (2010). Discrete time, finite state space mean field games. \textit{Journal de Math\'{e}matiques Pures et Appliqu\'{e}es}, \textit{93}, 308-328.

\bibitem{GuT} \textsc{Gu\'{e}ant, O.} (2009). \textit{Mean field games and applications to economics}. PhD Thesis, Universit\'{e} Paris-Dauphine.

\bibitem{Gu1} \textsc{Gu\'{e}ant, O.} (2009). A reference case for mean field games models. \textit{Journal de Math\'{e}matiques Pures et Appliqu\'{e}es}, \textbf{92}, 276-294.

\bibitem{Gu2} \textsc{Gu\'{e}ant, O.} (2010). Mean field games equation with quadratic Hamiltonian: a specific approach. \textit{Preprint}.

\bibitem{GLL} \textsc{Gu\'{e}ant, O., Lasry, J. M., Lions, P. L.} (2011). Mean field games and applications. In: A. R. Carmona et al. (eds.), \textit{Paris-Princeton Lectures on Mathematical Finance 2010}, Lecture Notes in Mathematics 2003, 205-266.

\bibitem{HCM1} \textsc{Huang, M., Caines, P.E., Malham\'{e}, R.P.} (2003). Individual and mass behaviour in large population stochastic wireless power control problems: centralized and Nash equilibrium solutions. \textit{Proceedings of the 42nd IEEE Conference on Decision and Control, Maui, Hawaii, December 2003}, 98 - 103.

\bibitem{HCM5} \textsc{Huang, M., Caines, P.E., Malham\'{e}, R.P.} (2004). Large-population cost-coupled LQG problems: generalizations to non-uniform individuals. \textit{Proceedings of the 43rd IEEE Conference on Decision and Control, Atlantis, Paradise Island, Bahamas, December 2004}, 3453-3458.

\bibitem{HCM4} \textsc{Huang, M., Caines, P.E., Malham\'{e}, R.P.} (2007a). An Invariance Principle in Large Population Stochastic Dynamic Games. \textit{Journal of Systems Science \& Complexity}, \textit{20} (2), 162-172.

\bibitem{HCM3} \textsc{Huang, M., Caines, P.E., Malham\'{e}, R.P.} (2007b). Large-population cost-coupled LQG problems with nonuniform agents: individual-mass behavior and decentralized $\epsilon$-Nash equilibria. \textit{IEEE Transactions on Automatic Control}, \textit{52} (9), 1560-1571.

\bibitem{HCM6} \textsc{Huang, M., Caines, P.E., Malham\'{e}, R.P.} (2009). Social optima in mean field LQG control: Centralized and decentralized strategies. \textit{Proceedings of the 47th Annual Allerton Conference, Allerton House, UIUC, Illinois, USA, September 2009}, 322-329.

\bibitem{HCM7} \textsc{Huang, M., Caines, P.E., Malham\'{e}, R.P.} (2010). Social certainty equivalence in mean field LQG control: social, Nash and centralized strategies. \textit{Proceedings of the 19th International Symposium on Mathematical Theory of Networks and Systems, Budapest, Hungary, July 2010}, 1525-1532.

\bibitem{HCM2} \textsc{Huang, M., Malham\'{e}, R.P., Caines, P.E.} (2006). Large population stochastic dynamic games: closed-loop McKean-Vlasov systems and the Nash certainty equivalence principle. \textit{Communications in Information and Systems}, \textit{6} (3), 221-252.

\bibitem{Ki} \textsc{Kirman, A.} (1993). Ants, Rationality, and Recruitment. \textit{The Quarterly Journal of Economics}, \textit{108} (1), 137-156.

\bibitem{Lac} \textsc{Lachapelle, A.} (2010). Human crowds and groups interactions: A mean field games approach. \textit{Preprint}.

\bibitem{LST} \textsc{Lachapelle, A., Salomon, J., Turinici, G.} (2010). Computation of mean field equilibria in economics. \textit{Mathematical Models and Methods in Applied Sciences}, \textbf{20} (4), 567-588.

\bibitem{LW} \textsc{Lachapelle, A., Wolfram M.-T.} (2011). On a mean field game approach modeling congestion and aversion in pedestrian crowds. \textit{Preprint}.

\bibitem{LL1} \textsc{Lasry, J.-M., Lions, P.-L.} (2006a). Jeux \'{a} champ moyen I - Le cas stationnaire. \textit{Comptes Rendus de l'Acad\'{e}mie des Sciences, Series I}, \textbf{343}, 619-625.

\bibitem{LL2} \textsc{Lasry, J.-M., Lions, P.-L.} (2006b). Jeux \'{a} champ moyen II. Horizon fini et contr\^{o}le optimal. \textit{Comptes Rendus de l'Acad\'{e}mie des Sciences, Series I}, \textbf{343}, 679-684.

\bibitem{LL3} \textsc{Lasry, J. M., Lions, P. L.} (2007). Mean field games. \textit{Japanese Journal of Mathematics} \textbf{2}(1), 229-260.

\bibitem{LZ} \textsc{Li, T., Zhang, J.-F.} (2008). Asymptotically optimal decentralized control for large population stochastic multiagent systems. \textit{IEEE Transactions on Automatic Control} \textbf{53} (7), 1643-1660.

\bibitem{MY} \textsc{Ma, J., Yong, J.} (1999). \textit{Forward - Backward Stochastic Differential Equations and Their Applications}, Lecture Notes in Mathematics, volume 1702, Springer, Berlin.

\bibitem{NCM} \textsc{Nourian, M., Caines, P.E., Malham\'{e}, R.P.} (2011). Mean Field Analysis of Controlled Cucker-Smale Type Flocking: Linear Analysis and Perturbation Equations. \textit{Proceedings of the 18th IFAC World Congress, Milan, August 2011}, 4471-4476.

\bibitem{Sz} \textsc{Sznitman, A. S.} (1989). Topics in propagation of chaos. In: D. L. Burkholder et al. (eds.), \textit{Ec\^{o}le de Probabilites de Saint Flour, XIX-1989}, Lecture Notes in Mathematics volume 1464, 165-251. 

\bibitem{TZB} \textsc{Tembine, H., Zhu, Q., Ba\c{s}ar, T.} (2011). Risk-sensitive mean-field stochastic differential games. \textit{Proceedings of the 18th IFAC World Congress, Milan, August 2011}, 3222-3227.

\bibitem{YMM} \textsc{Yang, T., Metha, P.G., Meyn, S.P.} (2011). A mean-field control-oriented approach to particle filtering. \textit{Proceedings of the American Control Conference (ACC), June 2011}, 2037-2043.

\bibitem{YMMS} \textsc{Yin, H., Mehta, P.G., Meyn, S.P., Shanbhag, U.V.} (2010). Synchronization of coupled oscillators is a game. \textit{Submitted to IEEE Transactions on Automatic Control}.

\end{thebibliography}
\end{document}